\documentclass[a4paper, 11pt]{article}
\usepackage[utf8]{inputenc}     
\usepackage[T1]{fontenc}
\usepackage{lmodern}
\usepackage{graphicx}
\usepackage[english]{babel}
\usepackage{subfigure}
\usepackage{amsmath, amsfonts, amssymb,amsthm}  
\usepackage{textcomp}
\usepackage{mathrsfs}
\usepackage{multirow}
\usepackage{float}
\usepackage{color}
\usepackage{epstopdf}

\theoremstyle{plain}
\newtheorem{theorem}{Theorem}[section]
\newtheorem{proposition}[theorem]{Proposition}
\newtheorem{lemma}[theorem]{Lemma}

\newtheorem{remark}[theorem]{Remark}
\newtheorem{definition}[theorem]{Definition}

\newcommand{\bydef}{\,\stackrel{\mbox{\tiny\textnormal{\raisebox{0ex}[0ex][0ex]{def}}}}{=}\,}
\newcommand{\CC}{\mathcal{C}}
\newcommand{\R}{\mathbb{R}}
\newcommand{\N}{\mathbb{N}}

\newcommand{\1}{\mathbf{1}}
\newcommand{\E}{\mathcal{E}}
\newcommand{\X}{\mathcal X}
\newcommand{\tC}{\tilde C}

\newcommand{\vect}[1]{{\boldsymbol{#1}}}
\newcommand{\bfr}{{\vect{r}}}

\title{Polynomial interpolation and a priori bootstrap for computer-assisted proofs in nonlinear ODEs}

\author{Maxime Breden \thanks{CMLA, ENS Cachan, CNRS, Université Paris-Saclay, 61 avenue du Pr\'esident Wilson, 94230 Cachan, France} \thanks{Département de Mathématiques et de Statistique, Université Laval, 1045 avenue de la Médecine, Québec, QC, G1V 0A6, Canada.}
\and Jean-Philippe Lessard\footnotemark[2] }

\date{}

\begin{document}

\maketitle

\begin{abstract}
In this work, we introduce a method based on piecewise polynomial interpolation to enclose rigorously solutions of nonlinear ODEs. Using a technique which we call {\em a priori bootstrap}, we transform the problem of solving the ODE into one of looking for a fixed point of a high order smoothing Picard-like operator. We then develop a rigorous computational method based on a Newton-Kantorovich type argument (the radii polynomial approach) to prove existence of a fixed point of the Picard-like operator. We present all necessary estimates in full generality and for any nonlinearities. With our approach, we study two systems of nonlinear equations: the Lorenz system and the ABC flow. For the Lorenz system, we solve Cauchy problems and prove existence of periodic and connecting orbits at the classical parameters, 
and for ABC flows, we prove existence of ballistic spiral orbits.
\end{abstract}

\begin{center}
{\bf \small Key words} 
{\small Computer-assisted proof $\cdot$ Nonlinear ODEs $\cdot$ Picard-like operator \\
Piecewise polynomial interpolation $\cdot$ Newton-Kantorovich $\cdot$ ABC flows}
\end{center}

\begin{center}
\end{center}

\section{Introduction}
\label{sec:intro}

This paper introduces an approach based on polynomial interpolation to obtain mathematically rigorous results about existence of solutions of nonlinear ordinary differential equations (ODEs). Our motivation for the present work is threefold. First, we believe that polynomial interpolation techniques are versatile and can lead to efficient and general computational methods to approximate solutions of ODEs with complicated (non polynomial) nonlinearities. Second, while polynomial interpolation techniques have be used to produce computer-assisted proofs in ODEs, their applicability to produce proofs is sometimes limited by the formulation of the problem itself. More precisely, a standard way to prove (both theoretically and computationally) existence of solutions of systems of ODEs is to reformulate the problem into an integral equation (often in the form of a Picard operator) and then to apply the contraction mapping theorem to get existence. If one is interested to produce computer-assisted proofs using that approach, the analytic estimates required to perform the proofs depend on the amount of regularity gained by applying the integral operator. This observation motivated developing what we call the {\em a priori bootstrap}, which consists of reformulating the original ODE problem into one of looking for the fixed point of a higher order smoothing Picard-like operator. Third, we believe (and hope) that our proposed method can be adapted to study infinite dimensional continuous dynamical systems (e.g. partial differential equations and delay differential equations) for which spectral methods may sometimes be difficult to apply (for instance because of the shape of the spatial domains or because the differential operators are difficult to invert in a given spectral basis). 

It is important to realize that computer-assisted arguments to study differential equations are by now standard, and that providing a complete overview of the literature would require a tremendous effort and is outside the scope of this paper. However, we encourage the reader to read the survey papers \cite{MR1962787,MR1420838,MR2652784,Jay_Konstantin_Survey,MR1849323,notices_jb_jp}, the books \cite{MR2807595,MR3467671} and to consult the webpage of the CAPD group \cite{capd} to get a flavour 
of the extraordinary recent advances made in the field. 

More closely related to the present work are methods based on the contraction mapping theorem using the \emph{radii polynomial approach} (first introduced in \cite{MR2338393}), which has been developed in the last decade to study fixed points, periodic orbits, invariant manifolds and connecting orbits of ODEs, partial differential equations and delay differential equations (see for instance \cite{MR3353132,BerShe15,MR3437754,CasLesMir16,MR3623202,MR2592879}). The numerics and a posteriori analysis in those works mainly use spectral methods like Fourier and Chebyshev series, and Taylor methods. First order polynomial (piecewise linear) approximations were also used using the radii polynomial approach (see \cite{MR3323206,MR2821596,MR3207723}), but more seldom, mainly because the numerical cost was higher and the accuracy was lower than for spectral methods. The computational cost of these low order methods is essentially due to the above mentioned low gain of regularity of the Picard operators chosen to perform the computer-assisted proofs. 

In an attempt to address the low gain of regularity problem, we present here a new technique that we call \emph{a priori bootstrap} which, when combined with the use of higher order interpolation, significantly improves the efficiency of computer-assisted proofs with polynomial interpolation methods. We stress that the limitations that affected the previous works using interpolation were not solely due to the use of first order methods, and that the \emph{a priori bootstrap} is crucial (that is, just increasing the order of the interpolation does not significantly improve the results in those previous works). This point is illustrated in Section~\ref{sec:applications}.

While we believe that one of the advantage of our proposed method is the versatility of the polynomial interpolations to 
tackle problems with complicated (non polynomial) nonlinearities, we hasten to mention the existence of previous powerful methods which have been developed in rigorous computing to study such problems. For instance, {\em automatic differentiation} (AD) techniques provide a beautiful and efficient means of computing solutions of nonlinear problems (e.g. see \cite{MR2807595,MR633878,MR2146523}) and are often combined with Taylor series expansions to prove existence of solutions of differential equations with non polynomial nonlinearities (e.g. see \cite{MR1962787, MR1652147, MR2644324,MR1930946, MR1961956, MR2049869,Ru99a}). Also, in the recent work \cite{MR3545977}, the ideas of AD are combined with Fourier series to prove existence of periodic orbits in non polynomial vector fields. Independently of AD techniques, a method involving Chebyshev series to approximate the nonlinearities have been proposed recently in \cite{MR3124898}. Finally, the fast Fourier transform (FFT) algorithm is used in \cite{FHL_KAM} to control general nonlinearities in the context of computer-assisted proofs in KAM theory.

In this paper we  consider $\phi:\R^n\to\R^n$ a $C^{\bfr}$ vector field (not necessarily polynomial) with ${\bfr} \ge 1$, and we present a rigorous numerical procedure to study problems of the form
\begin{equation}
\label{eq:general_problem}
\left\{
\begin{aligned}
&\frac{du}{dt}(t)=\phi(u(t)), \quad t\in[0,\tau],\\
&BV(u(0),u(\tau))=0.
\end{aligned}
\right.
\end{equation}

We treat three special cases for $BV$, corresponding to an initial value problem, a problem of looking for periodic orbits and a problem of looking for connecting orbits. We also note that, as for the already existing spectral methods, the technique presented here extends easily to treat parameter dependent versions of~\eqref{eq:general_problem} (e.g. using ideas from \cite{MR2630003,MR3125637,MR3464215}). For the sake of simplicity, we expose all the general arguments for the initial value problem only, that is for
\begin{equation}
\label{eq:general_IVP}
\left\{
\begin{aligned}
&\frac{du}{dt}(t)=\phi(u(t)), \quad t\in[0,\tau],\\
&u(0)=u_0,
\end{aligned}
\right.
\end{equation}
given an integration time $\tau > 0$ and an initial condition $u_0 \in \R^n$. We explain how~\eqref{eq:general_IVP} needs to be modified for different problems as we introduce them in Sections~\ref{sec:applications} and~\ref{sec:ABC}.

Our paper is organized as follows. In Section~\ref{sec:a_priori_bootstrap}, we start by presenting our \emph{a priori bootstrap} technique, together with a \emph{piecewise} reformulation of the operator that we use throughout this work. We then recall in Section~\ref{sec:general} some definitions and error estimates about polynomial interpolation, and explain how to combine them with our \emph{a priori bootstrap} formulation to get computer-assisted proofs. The precise estimates needed for the proofs are then derived in Section~\ref{sec:bounds}, and their dependency with respect to the \emph{a priori bootstrap} and to the parameters of the polynomial interpolation is commented in Section~\ref{sec:para}. This discussion is complemented by several examples in Section~\ref{sec:applications}, where we apply our technique to validate solutions for the Lorenz system. Finally we give another example of application in Section~\ref{sec:ABC}, where we prove the existence of some specific orbits for ABC flows.

\section{Reformulations of the Cauchy problem}
\label{sec:a_priori_bootstrap}
\subsection{A priori bootstrap}

One of the usual strategies used to study~\eqref{eq:general_IVP}, both theoretically and numerically, is to recast it as a fixed point problem, as in the following lemma.
\begin{lemma}
Consider the standard Picard operator
\begin{equation*}
f:\left\{
\begin{aligned}
\mathcal{C}^0([0,1],\R^n) &\to \mathcal{C}^1([0,1],\R^n) \\
u &\mapsto\ f(u),
\end{aligned}
\right.
\end{equation*}
where
\begin{equation} \label{eq:standard_Picard}
f(u)(t) \bydef u_0+\tau\int_0^t \phi(u(s))ds,\quad \text{for all}~t\in [0,1].
\end{equation}
Then $u$ is a fixed point of $f$ if and only if $v:t\mapsto u(\frac{t}{\tau})$ is a solution of~\eqref{eq:general_IVP}.
\end{lemma}
In previous works using this reformulation, the limiting factor in the estimates needed to apply the contraction mapping theorem was a consequence of the fact that $f$ only gains one order of regularity, that is maps $\CC^0$ into $\CC^1$. This fact will be made precise in Section~\ref{sec:bounds} where we derive the estimates in question and in Section~\ref{sec:para} where we discuss how those estimates affect the effectiveness of our technique.

To circumvent this limitation, we propose a different reformulation that we call \emph{a priori bootstrap}. This approach provides operators which gain more regularity, and therefore lead to sharper analytic estimates. First we introduce some notations. The following definition allows concisely describing the higher order equations obtained by taking successive derivatives of~\eqref{eq:general_IVP}.

\begin{definition}
Consider the sequence of vector fields $\left( \phi^{[p]} \right)_{0\leq p\leq \bfr+1}$ with $\phi^{[p]} : \R^n \to \R^n$,
\begin{equation*}
\phi^{[0]}(u) \bydef u\quad \text{and}\quad \phi^{[p+1]}(u) \bydef D\phi^{[p]}(u)\phi(u)\quad \text{for all } u \in \R^n \text{ and for all } p =  0,\dots,\bfr.
\end{equation*}
\end{definition}
\begin{lemma}
For any $1\leq p\leq \bfr+1$, $u$ solves~\eqref{eq:general_IVP} if and only if $u$ solves the following Cauchy problem
\begin{equation}
\label{eq:order_p_ODE}
\left\{
\begin{aligned}
& \frac{d^pu}{dt^p}(t)=\phi^{[p]}(u(t)), \quad t\in[0,\tau],\\
& \frac{d^qu}{dt^q}(0)=\phi^{[q]}(u_0), \quad \text{for all } q = 0,\dots,p-1.
\end{aligned}
\right.
\end{equation}
\end{lemma}
\begin{proof}
The direct implication is trivial. To prove the converse application, we consider $e \bydef \frac{du}{dt}-\phi(u)$ and show that is solve a linear ODE of order $p-1$, with initial conditions $\frac{d^qe}{de^q}(0)=0$ for all $q=0,\dots,p-2$, which implies that $e \equiv 0$.
\end{proof}
Integrating the $p^{th}$ order Cauchy problem \eqref{eq:order_p_ODE} $p$ times leads to a new fixed point operator which now maps $\CC^0$ into $\CC^p$. 
\begin{lemma}
\label{lem:operator_order_p}
Let $1\leq p\leq \bfr+1$ and consider the Picard-like operator 
\begin{equation*}
\tilde g:\left\{
\begin{aligned}
\mathcal{C}^0([0,1],\R^n) &\to \mathcal{C}^p([0,1],\R^n) \\
u &\mapsto\ \tilde g(u),
\end{aligned}
\right.
\end{equation*}
where
\begin{equation} \label{eq:tilde_g_original}
\tilde g(u)(t) \bydef \sum_{q=0}^{p-1}\tau^q\frac{t^q}{q!} \phi^{[q]}(u_0)+\tau^p\int_0^t \frac{(t-s)^{p-1}}{(p-1)!}\phi^{[p]}(u(s))ds,\quad \text{for all}~t\in [0,1].
\end{equation}
Then $u$ is a fixed point of $\tilde g$ if and only if $v:t\mapsto u(\frac{t}{\tau})$ is a solution of~\eqref{eq:order_p_ODE} (and thus of~\eqref{eq:general_IVP}).
\end{lemma}
\begin{proof}
If $u$ is a fixed point of $\tilde g$, an elementary computation yields that $v$ solves~\eqref{eq:order_p_ODE}. Conversely, if $v$ solves~\eqref{eq:order_p_ODE} then Taylor's formula with integral reminder shows that $u$ is a fixed point of $\tilde g$.
\end{proof}

It is worth noting that, in the same framework of rigorous computation as the one used here, approximations using piecewise linear functions were used in \cite{MR2821596} to prove existence of homoclinic orbits for the Gray-Scott equation. In that case the system of ODEs considered is of order $2$, and therefore the equivalent integral operator is very similar to $\tilde g$ in \eqref{eq:tilde_g_original} for $p=2$. Similarly, piecewise linear functions were used in \cite{MR3207723} to prove existence of connecting orbits in the Lorenz equations. In that case, the standard Picard operator \eqref{eq:standard_Picard} was used.

Now that we have an operator which provides a gain of several orders of regularity, it becomes interesting to consider polynomial interpolation of higher order, and again this will be detailed in Section~\ref{sec:para} and applied in Section~\ref{sec:applications}.

\subsection{Piecewise reformulation of the Picard-like operator}
\label{sec:reformulation}

We finish this section by a last equivalent formulation of the initial value problem~\eqref{eq:general_IVP}, that will be the one used in the present paper to perform the computer-assisted proofs. Given $m\in\N$, we introduce the mesh of $[0,1]$
\begin{equation*}
\Delta_m\bydef \{t_0,t_1,\ldots,t_m\},
\end{equation*}
where $t_0=0<t_1<\ldots<t_m=1$. Then we consider $\CC^0_{\Delta_{m}}([0,1],\R^n)$ (respectively $\CC^k_{\Delta_{m}}([0,1],\R^n)$) the space of piecewise continuous (respectively $\CC^k)$ functions having possible discontinuities only on the mesh $\Delta_m$. More precisely, we use the following definition.
\begin{definition}
For $k\in\N$, we say that $u\in\CC^k_{\Delta_{m}}([0,1],\R^n)$ if $u_{ |_{(t_{j},t_{j+1})}}\in\CC^k((t_{j},t_{j+1}),\R^n)$ and can be extended to a $\CC^k$ function on $[t_j,t_{j+1}]$, for all $j=0,\ldots,m-1$.
\end{definition}
We then introduce
\[
g:\left\{
\begin{aligned}
\CC^0_{\Delta_m}([0,1],\R^n) &\to \CC^p_{\Delta_m}([0,1],\R^n) \\
u &\mapsto g(u),
\end{aligned}
\right.
\]
defined on the interval $(t_j,t_{j+1})$ ($j=0,\dots,m-1$) by
\begin{equation}  \label{eq:piecewise_g}
g(u)(t) \bydef  \sum_{q=0}^{p-1}\tau^q\frac{(t-t_j)^q}{q!} \phi^{[q]}(u(t_j^-))+\tau^p\int_{t_j}^t \frac{(t-s)^{p-1}}{(p-1)!}\phi^{[p]}(u(s))ds,
\end{equation}
where $u(t_j^-)$ denotes the left limit of $u$ at $t_j$, and $u(t_0^-)$ must be replaced by $u_0$ (this last convention will be used throughout the paper).

\begin{remark} \label{rem:piecewise_g}
{\em 
We point out that our computer-assisted proof is based on the operator $g$ (defined in~\eqref{eq:piecewise_g}), which differs slightly from the operator $\tilde g$ (defined in~\eqref{eq:tilde_g_original}), which was used in previous studies such as~\cite{MR3323206,MR2821596,MR3207723}. The only difference is that the integral in $g$ is in some sense reseted at each $t_j$. We introduce this \emph{piecewise} reformulation because it allows for sharper estimates (see Remark~\ref{rem:sharper_estimates}).
}
\end{remark}
 
We finally introduce $G:\CC^0_{\Delta_m}([0,1],\R^n) \to \CC^0_{\Delta_m}([0,1],\R^n)$ as
\begin{equation*}
G(u) \bydef g(u)-u.
\end{equation*} 
\begin{lemma}
Let $u\in\CC^0_{\Delta_m}([0,1],\R^n)$. Then $G(u)=0$ if and only if $v:t\mapsto u(\frac{t}{\tau})$ solves~\eqref{eq:general_IVP}.
\end{lemma}
\begin{proof}
This result is similar to Lemma~\ref{lem:operator_order_p}. The only additional property that we need to check is that, if $u\in\CC^0_{\Delta_m}([0,1],\R^n)$ satisfies $G(u)=0$, then $u$ cannot be discontinuous. Indeed, if $G(u)=0$ then $g(u)=u$, and for all $j\in\{1,\ldots,m-1\}$ one has
\begin{equation*}
u(t_j^-)=\lim\limits_{t\to t_j^+} g(u)(t) = \lim\limits_{t\to t_j^+} u(t) = u(t_j^+). \qedhere
\end{equation*}
\end{proof} 

At this point, it might seems as if defining $G$ on $\CC^0_{\Delta_m}([0,1],\R^n)$ brings unnecessary complications, and that we should simply define it on $\CC^0([0,1],\R^n)$. While this is indeed a possibility, it will quickly become apparent that the present choice is more convenient, both for theoretical and numerical considerations (see Remark~\ref{rem:why_discontinuous}).

Finding a zero of $G$ is the formulation of our initial problem~\eqref{eq:general_IVP} that we are going to use in the rest of this paper.

\section{General framework for the polynomial interpolation}
\label{sec:general}

\subsection{Preliminaries}

Given a mesh $\Delta_m$ as defined in Section~\ref{sec:reformulation}, we introduce the refined mesh $\Delta_{m,k}$ where, for all $j\in\{0,\ldots,m-1\}$ we add $k-1$ points between $t_{j}$ and $t_{j+1}$. More precisely we suppose that these points are the Chebyshev points (of the second kind) between $t_{j}$ and $t_{j+1}$, that is we add the following points:
\begin{equation*}
t_{j,l} \bydef t_{j}+\frac{x_l^k+1}{2}(t_{j+1}-t_{j}), \quad \text{for } l = 1,\dots,k-1,
\end{equation*}
where
\begin{equation*}
 x_l^k \bydef \cos\theta_l^k,\quad
\theta_l^k \bydef \frac{k-l}{l}\pi,\quad \text{for } l = 0, \dots,k.
\end{equation*}
Notice that the above definition extends to $t_{j,0}=t_{j}$ and $t_{j,k}=t_{j+1}$, and that $k=1$ corresponds to the mesh used in previous studies with first order interpolation (e.g. see \cite{MR2821596,MR3207723,MR3323206}).

We then introduce the subspace $S_{m,k}^n \subset \CC^0_{\Delta_{m}}([0,1],\R^n)$ of piecewise polynomial functions of degree $k$ on $\Delta_m$
\begin{equation*}
S_{m,k}^n\bydef  
\left\{ 
u\in\CC^0_{\Delta_{m}}([0,1],\R^n): \ u_{ |_{(t_{j},t_{j+1})}} \text{ is a polynomial of degree } k \text{ for all } j=0,\dots,m-1 
\right\}.
\end{equation*}
Next, we define the projection operator
\begin{equation*}
\Pi_{m,k}^n : \left\{
\begin{aligned}
\CC^0_{\Delta_{m}}([0,1],\R^n) &\to S_{m,k}^n \\
u &\mapsto \bar u = \Pi_{m,k}^n(u),
\end{aligned}
\right.
\end{equation*}
where $\bar u$ is the function in $S_{m,k}^n$ that matches the values of $u$ on the mesh $\Delta_{m,k}$. Notice that $u$ can have discontinuities at the points $t_j$, therefore the matching of $u$ and $\bar u$ at those points must be understood as
\begin{equation*}
u(t_j^-)=\bar u(t_j^-) \quad \text{and}\quad u(t_j^+)=\bar u(t_j^+).
\end{equation*} 

In the sequel we will need to control the error between a function $u$ and its interpolation $\bar u$. This is the content of the following propositions, where $\left\Vert\cdot\right\Vert_{\infty}$ denotes the sup norm on $[0,1]$.
\begin{proposition}
\label{prop:interpolation_error1}
For all $u\in\CC^{k+1}_{\Delta_{m}}([0,1],\R)$,
\begin{equation*}
\left\Vert (Id-\Pi^1_{m,k})u\right\Vert_{\infty}=\left\Vert u-\bar u\right\Vert_{\infty} \leq C_{k}\max\limits_{0\leq j<m}\left((t_{j+1}-t_{j})^{k+1}\max\limits_{t\in[t_{j},t_{j+1}]}\left\vert \frac{d^{k+1} u}{dt^{k+1}}(t)\right\vert\right),
\end{equation*}
where
\begin{equation*}
C_{k} \bydef \frac{1}{(k+1)!2^{2k}}.
\end{equation*}
\end{proposition}
\begin{proposition}
\label{prop:interpolation_error2}
Fix $l\in\N$ such that $1\leq l\leq k$. For all $u\in\CC^l_{\Delta_{m}}([0,1],\R)$,
\begin{equation*}
\left\Vert (Id-\Pi^1_{m,k})u\right\Vert_{\infty}=\left\Vert u-\bar u\right\Vert_{\infty} \leq \tC_{k,l}\max\limits_{0\leq j<m}\left((t_{j+1}-t_{j})^l \max\limits_{t\in[t_{j},t_{j+1}]}\left\vert \frac{d^{l} u}{dt^{l}}(t)\right\vert\right),
\end{equation*}
where
\begin{equation*}
\tC_{k,l} \bydef \min\left[\left(1+\Lambda_k\right)\left(\frac{\pi}{4}\right)^l\frac{(k+1-l)!}{(k+1)!},\frac{1}{l!2^l}\sum_{q=0}^{\left[\frac{l-1}{2}\right]}\frac{1}{4^q}\binom{l-1}{2q}\binom{2q}{q}\right],
\end{equation*}
$\Lambda_k$ being the Lebesgue constant (see for instance~\cite{MR3012510}), and $\left[\frac{l-1}{2}\right]$ denoting the integer part of $\frac{l-1}{2}$.
\end{proposition}

\begin{remark}
{\em 
More information about the Lebesgue constant, and in particular sharp computable upper bounds for it, can be found in the Appendix, together with references and proofs of the two above propositions.
}
\end{remark}

\subsection{Finite dimensional projection}
\label{sec:numerical_implementation}

To get an approximate zero of $G$ (and thus an approximate solution of~\eqref{eq:general_IVP}), we are going to look for a zero of $\bar G\bydef \Pi^n_{m,k} G_{ |S^n_{m,k}}$. But first, we need a convenient way to represent the elements of $S^n_{m,k}$. Here and in the sequel, we use the exponent ${}^{(i)}$ to denote the $i$-th component of a vector in $\R^n$, but we will work with all the components at once as often as possible to avoid burdening the notations with this exponent ${}^{(i)}$. Let us introduce the set of indexes
\begin{equation*}
\E^n_{m,k} \bydef \left\lbrace (i,j,l)\in\N^3,\ 1\leq i\leq n,\ 0\leq j\leq m-1,\ 0\leq l\leq k\right\rbrace.
\end{equation*} 
Perhaps the most natural way to characterize an element $\bar u$ of $S^n_{m,k}$ is to give all the values $\bar u^{(i)}(t_{j,l})$ for $(i,j,l)\in \E^n_{m,k}$.  However, we will also use another representation, more suited to numerical computations, which consists of decomposing $\bar u$ on the Chebyshev basis. That is, we write
\begin{equation}
\label{eq:cheb_representation}
\bar u^{(i)}(t)=\sum_{l=0}^k \bar u^{(i)}_{j,l} T_{l}\left(\frac{t-t_j}{t_{j+1}-t_j}-\frac{t_{j+1}-t}{t_{j+1}-t_j}\right),\quad \text{for all } j = 0,\dots,m-1 \text{ and } t \in (t_j,t_{j+1}),
\end{equation}
where $T_{l}$ is the $l$-th Chebyshev polynomial of the first kind. We can thus also describe uniquely any function $u$ belonging to $S^n_{m,k}$ by the family of Chebyshev coefficients $\left(\bar u^{(i)}_{j,l}\right)_{(i,j,l)\in \E^n_{m,k}}$.
\begin{remark} \label{rem:why_discontinuous}
{\em
Let us mention how considering functions with possible discontinuities on the mesh points in $\Delta_m$ comes in handy. By restricting ourselves to functions in $\CC^0([0,1],\R^n)$, we would need additional constraints on the Chebyshev coefficients to impose the continuity at each of the mesh point $t_j$ ($j = 1,\dots,m-1$) and keep track of them in all computations. Instead, the choice of working with $\CC^0_{\Delta_{m}}([0,1],\R^n)$ allows avoiding these additional constraints.
}
\end{remark}
For $u\in\CC^0_{\Delta_{m}}([0,1],\R^n)$ we have
\begin{align}
\label{eq:projected_G}
&G(u)\left(t_{j,l}\right) = \sum_{q=0}^{p-1}\tau^q\frac{(t_{j,l}-t_j)^q}{q!} \phi^{[q]}(u(t_j^-))+\tau^p\int_{t_j}^{t_{j,l}} \frac{(t_{j,l}-s)^{p-1}}{(p-1)!} \phi^{[p]}(u(s))ds -u(t_{j,l}), \nonumber\\
& \hspace{8.3cm} \text{for all}~j = 0,\dots,m-1 \text{ and}~l = 0,\dots,k.
\end{align}
We recall that all the values $G^{(i)}(u)\left(t_{j,l}\right)$ for $(i,j,l)\in\E^n_{k,m}$ uniquely characterize $\Pi^n_{m,k}G(u)$.

Using the isomorphisms $\bar u\mapsto \left(\bar u^{(i)}(t_{j,l})\right)_{(i,j,l)\in \E^n_{m,k}}$ and $\bar u\mapsto \left(\bar u^{(i)}_{j,l}\right)_{(i,j,l)\in \E^n_{m,k}}$ to identify $S^n_{m,k}$ and $\R^{nm(k+1)}$, we can in fact see $\bar G\bydef \Pi^n_{m,k} G_{ |S^n_{m,k}}$ as a function from $\R^{nm(k+1)}$ to itself, that associates to the coefficients $\left(\bar u^{(i)}_{j,l}\right)_{(i,j,l)\in \E^n_{m,k}}$ the values $\left(G^{(i)}(\bar u)(t_{j,l})\right)_{(i,j,l)\in \E^n_{m,k}}$. Thus we can numerically find a zero $\bar u$ of $\bar G$, which is going to be our approximate solution. We note that we use these identifications between $S^n_{m,k}$ and $\R^{nm(k+1)}$ throughout the present work. Our objective is now to \emph{validate} this numerical solution $\bar u$, that is to prove that within a given neighbourhood of $\bar u$ lies a true zero $u$ of $G$.

\subsection{Back to a fixed point formulation}
\label{sec:operator}

We consider the space $\X^n \bydef \CC^0_{\Delta_{m}}([0,1],\R^n)$ and its decomposition $\X^n=\X^n_{m,k}\oplus \X^n_{\infty}$, where
\begin{equation*}
\X^n_{m,k} \bydef S_{m,k}^n \quad \text{and}\quad  
\X^n_{\infty} \bydef (Id-\Pi_{m,k}^n)\CC^0_{\Delta_{m}}([0,1],\R^n).
\end{equation*}
We already have a projection onto $\X^n_{m,k}$
\begin{equation*}
\Pi^n_{m,k}:\left\{
\begin{aligned}
\X^n &\to \X^n_{m,k} \\
u &\mapsto \bar u = \Pi^n_{m,k}(u),
\end{aligned}
\right.
\end{equation*}
and we also define its complementary
\begin{equation*}
\Pi^n_{\infty}:\left\{
\begin{aligned}
\X^n &\to \X^n_{\infty} \\
u &\mapsto \Pi^n_{\infty}(u) =  u-\bar u = (Id - \Pi^n_{m,k}) (u) .
\end{aligned}
\right.
\end{equation*}

We then define the norms
\begin{equation*}
\left\Vert \Pi^n_{m,k}(u)\right\Vert_{\X^n_{m,k}}\bydef \max\limits_{(i,j,l)\in\E^n_{m,k}}\left\vert \bar u^{(i)}(t_{j,l})\right\vert \quad \text{and} \quad \left\Vert \Pi^n_{\infty}(u)\right\Vert_{\X^n_{\infty}}\bydef \max\limits_{1\leq i\leq n}\left\Vert u^{(i)}-\bar u^{(i)}\right\Vert_{\infty}.
\end{equation*}
On $\X^n$ we consider the norm
\begin{equation*}
\left\Vert u\right\Vert_{\X^n} \bydef \max\left(\left\Vert \Pi^n_{m,k}(u)\right\Vert_{\X^n_{m,k}},\frac{1}{r_{\infty}}\left\Vert \Pi^n_{\infty}(u)\right\Vert_{\X^n_{\infty}}\right),
\end{equation*}
where $r_{\infty}$ is a positive parameter. Notice that for all $r_{\infty}>0$, $\left(\X^n,\Vert\cdot\Vert_{\X^n}\right)$ is a Banach space. For any $r,r_{\infty}>0$, we denote by $B_{\X^n}(r,r_{\infty})$ the closed neighbourhood of $0$ defined as
\begin{equation*}
B_{\X^n}(r,r_{\infty}) = \left\lbrace u\in\X, \left\Vert u\right\Vert_{\X^n}\leq r \right\rbrace.
\end{equation*}

Suppose that we now have computed a numerical zero $\bar u$ of $\bar G$. We define $A_{m,k}^{\dag}=D\bar G\left(\bar u\right)$ and consider $A_{m,k}$ an injective numerical inverse of $A_{m,k}^{\dag}$. Finally, we introduce the \emph{Newton-like} operator $T:\X^n\to \X^n$ defined by
\begin{equation*}
T(u) \bydef \left(\Pi^n_{m,k}-A_{m,k}\Pi^n_{m,k}G\right)u + \Pi^n_{\infty}\left(G(u)+u\right).
\end{equation*}
Notice that the fixed points of $T$ are in one-to-one correspondence with the zeros of $G$. We now give a finite set of sufficient conditions, that can be rigorously checked on a computer using interval arithmetic, to ensure that $T$ is a contraction on a given ball around $\bar u$. If those conditions are satisfied, the Banach fixed point theorem then yields the existence and local uniqueness of a zero of $G$. This is the content of the following statement (based on \cite{MR1639986}, see also \cite{MR2338393} for a detailed proof).

\begin{theorem}
\label{th:rad_pol}
Let
\begin{equation}
\label{eq:def_y}
y \bydef T (\bar u) - \bar u,
\end{equation}
and
\begin{equation}
\label{eq:def_z}
z=z(u_1,u_2)\bydef DT(\bar u+u_1)u_2, \quad \text{for all}~u_1,u_2\in B_{\X^n}(r,r_{\infty}).
\end{equation}
Assume that we have bounds $Y$ and $Z(r,r_{\infty})$ satisfying
\begin{equation}
\label{eq:condition_Y}
\left\vert \left(\Pi^n_{m,k} y\right)^{(i)}_{j,l} \right\vert \leq Y^{(i)}_{j,l},\quad \text{for all}~(i,j,l)\in\E^n_{m,k},
\end{equation}
\begin{equation}
\label{eq:condition_Yinfty}
\left\Vert \left(\Pi^n_{\infty} y\right)^{(i)}\right\Vert_{\infty} \leq Y^{(i)}_{\infty},\quad \text{for all}~1\leq i\leq n,
\end{equation}
\begin{equation}
\label{eq:condition_Z}
\sup\limits_{u_1,u_2\in B_{\X^n}(r,r_{\infty})} \left\vert \left(\Pi^n_{m,k} z(u_1,u_2)\right)^{(i)}_{j,l} \right\vert \leq Z^{(i)}_{j,l}(r,r_{\infty}),\quad \text{for all}~(i,j,l)\in\E^n_{m,k},
\end{equation}
and
\begin{equation}
\label{eq:condition_Zinfty}
\sup\limits_{u_1,u_2\in B_{\X^n}(r,r_{\infty})} \left\Vert \left(\Pi^n_{\infty} z(u_1,u_2)\right)^{(i)}\right\Vert_{\infty} \leq Z^{(i)}_{\infty}(r,r_{\infty}),\quad \text{for all}~1\leq i\leq n.
\end{equation}
If there exist $r,r_{\infty}>0$ such that
\begin{equation}
\label{eq:condition_p_finite}
p^{(i)}_{j,l}(r,r_{\infty})\bydef Y^{(i)}_{j,l}+Z^{(i)}_{j,l}(r,r_{\infty})-r<0,\quad \text{for all}~(i,j,l)\in\E^n_{m,k}
\end{equation}
and
\begin{equation}
\label{eq:condition_p_infty}
p^{(i)}_{\infty}(r,r_{\infty})\bydef Y^{(i)}_{\infty}+Z^{(i)}_{\infty}(r,r_{\infty})-r_{\infty} r<0,\quad \text{for all}~1\leq i\leq n,
\end{equation}
then there exists a unique zero of $G$ within the set $\bar u + B_{\X^n}(r,r_{\infty}) \subset \X^n$.
\end{theorem}

The quantities $p^{(i)}_{j,l}(r,r_{\infty})$ and $p^{(i)}_{\infty}(r,r_{\infty})$ given respectively in 
\eqref{eq:condition_p_finite} and \eqref{eq:condition_p_infty} are called the {\em radii polynomials}.

In the next section, we show how to obtain bounds $Y$ and $Z$ satisfying~\eqref{eq:condition_Y}-\eqref{eq:condition_Zinfty}. Before doing so, let us make a quick remark about the different representations and norms we can use on $\X^n_{m,k}$.

\begin{remark}
{\em
As explained in Section~\ref{sec:numerical_implementation}, in practice we will mostly work with $\bar u\in \X^n_{m,k}$ represented by its Chebyshev coefficients as in~\eqref{eq:cheb_representation}. However, there are going to be instances where the values $\bar u^{(i)}(t_{j,l})$ are needed, for instance to compute $\left\Vert \bar u\right\Vert_{\X^n_{m,k}}$. We point out that numerically, passing from one representation to the other can be done easily by using the Fast Fourier Transform.

One other important point is that, at some point in the next section we are going to need upper bounds for $\left\Vert \bar u^{(i)}\right\Vert_{\infty}$. To get such a bound from our finite dimensional data, we have two options, namely
\begin{equation}
\label{eq:bound_infty_from_sum}
\max\limits_{t\in[t_j,t_{j+1}]}\left\vert \bar u^{(i)}(t)\right\vert \leq \sum_{l=0}^k \left\vert \bar u^{(i)}_{j,l} \right\vert,\quad \text{for all}~j=0,\dots,m-1,
\end{equation}
or
\begin{equation}
\label{eq:bound_infty_from_max}
\max\limits_{t\in[t_j,t_{j+1}]}\left\vert \bar u^{(i)}(t)\right\vert \leq \Lambda_k \max\limits_{0\leq l\leq k} \left\vert \bar u^{(i)}(t_{j,l}) \right\vert,\quad \text{for all}~j = 0,\dots,m-1.
\end{equation}
If $\bar u$ is given, then~\eqref{eq:bound_infty_from_sum} is usually better, whereas~\eqref{eq:bound_infty_from_max} is better if $\bar u$ is any function in a given ball of $\X^n_{m,k}$. Notice that~\eqref{eq:bound_infty_from_sum} simply follows from the fact that the Chebyshev polynomials satisfy $\vert T_l(t)\vert \leq 1$ for all $t\in[-1,1]$ and all $l\in\N$. For more information about the bound~\eqref{eq:bound_infty_from_max}, see the Appendix and the references therein.
}
\end{remark}

\section{Formula for the bounds}
\label{sec:bounds}

In this section, we give formulas for $Y^{(i)}_{j,l}$, $Y^{(i)}_{\infty}$,  $Z^{(i)}_{j,l}$ and $Z^{(i)}_{\infty}$ satisfying the assumptions~\eqref{eq:condition_Y}-\eqref{eq:condition_Zinfty} of Theorem~\ref{th:rad_pol}. To make the exposition clearer, we focus strictly on the derivation of the different bounds in this section. In particular, the discussion about the impact of the level of an priori bootstrap (that is the value of $p$) and the order of polynomial approximation (that is the value of $k$) is done in Section~\ref{sec:para}.

\subsection{The \boldmath$Y$\unboldmath~bounds}

In this section we derive the $Y$ bounds, which measure the \emph{defect} associated with a numerical solution $\bar u$, that is how close $G(\bar u)$ is to $0$. We start by the \emph{finite dimensional} part.

\begin{proposition}
\label{prop:Y_finite}
Let $y$ be defined as in~\eqref{eq:def_y} and consider
\begin{equation*}
Y^{(i)}_{j,l} \geq \left\vert \left( A_{m,k}\bar G(\bar u) \right)^{(i)}_{j,l} \right\vert, \quad \text{for all}~(i,j,l)\in\E^n_{m,k},
\end{equation*}
where $\bar G(\bar u)$ is here seen as the vector $\left(G^{(i)}(\bar u)(t_{j,l})\right)_{(i,j,l)\in \E^n_{m,k}}$. Then~\eqref{eq:condition_Y} holds.
\end{proposition}
\begin{proof}
Simply notice that $\Pi^n_{m,k} y=-A_{m,k}\Pi^n_{m,k} G(\bar u)$.
\end{proof}

\begin{remark} \label{rem:error_integrals}
{\em
The above bound is not completely satisfactory, in the sense that is not directly implementable. Indeed, to compute $Y^{(i)}_{j,l}$ we need to evaluate (or at least to bound) the quantities $G^{(i)}(\bar u)(t_{j,l})$. In particular (see~\eqref{eq:projected_G}), we need to evaluate the integrals
\begin{equation*}
\int_{t_j}^{t_{j,l}} (t_{j,l}-s)^{p-1} \phi^{[p]}(\bar u(s))ds=\left(\frac{t_{j+1}-t_j}{2}\right)^p\int_{-1}^{x^k_l}(x^k_l-s)^p\Psi(s)ds,
\end{equation*}
where
\begin{equation*}
\Psi(s) \bydef \phi^{[p]}\left(\sum_{l=0}^{k}\bar u_{j,l} T_l(s)\right).
\end{equation*}

If $\phi$ is a non polynomial vector field, we use a Taylor approximation of order $k_0$ of $\Psi$ to get an approximate value of the integral by computing
\begin{equation*}
\sum_{l=0}^{k_0}\frac{1}{l!}\frac{d^l\Psi}{ds^l}(0)\int_{-1}^{x^k_l}(x^k_l-s)^ps^lds.
\end{equation*}
Notice that this quantity can be evaluated explicitly. The error made in this approximation is then controlled as follows
\begin{align*}
&\left\Vert \int_{-1}^{x^k_l}(x^k_l-s)^p\Psi(s)ds - \sum_{l=0}^{k_0}\frac{1}{l!}\frac{d^l\Psi}{ds^l}(0)\int_{-1}^{x^k_l}(x^k_l-s)^ps^lds \right\Vert \leq& \\
&\hspace{7cm} \frac{1}{(k_0+1)!}\max_{s\in[-1,1]}\left\| \frac{d^{k_0+1}\Psi}{ds^{k_0+1}}(s)\right\| \int_{-1}^{x_l^k}(x_l^k-s)^p\vert s\vert^{k_0+1}ds.
\end{align*}
Notice that this error term is effective, since $\max_{s\in[-1,1]}\left\Vert\frac{d^{k_0+1}\Psi}{ds^{k_0+1}}(s)\right\Vert$ can be bounded using interval arithmetic. Therefore, the quantity $Y^{(i)}_{j,l}$ that we end up implementing is of the form
\begin{equation*}
Y^{(i)}_{j,l} = \left| \left( A_{m,k}\hat G(\bar u) \right)^{(i)}_{j,l} \right| + \left\vert \left( \left\vert A_{m,k}\right\vert G_{\epsilon}(\bar u) \right)^{(i)}_{j,l} \right\vert,
\end{equation*}
where the vector $\hat G(\bar u)$ contains the approximate integrals and the vector $G_{\epsilon}(\bar u)$ contains the errors bounds for these approximations. Here and in the sequel, absolute values applied to a matrix, like $\vert A_{m,k}\vert$, must be understood component-wise. We point out that, in practice, if the mesh $\Delta_m$ is refined enough, then $\bar u$ is not going to be varying much on each subinterval $[t_j,t_{j+1}]$, and thus we can get rather precise approximations even with a lower order $k_0$ for the Taylor expansion.

We mention that when the vector field $\phi$ is polynomial, $\Psi$ has a finite Taylor expansion, therefore up the integrals can in fact be computed exactly (i.e. we can get $G_{\epsilon}(\bar u)=0$).
}
\end{remark}

We now turn our attention to the second part of the $Y$ bound.

\begin{proposition}
\label{prop:Y_infty}
Let $y$ be defined as in~\eqref{eq:def_y} and consider
\begin{equation*}
Y^{(i)}_{\infty} \geq C_k \tau^p \max\limits_{0\leq j<m}\left((t_{j+1}-t_{j})^{k+1}\max\limits_{t\in[t_{j},t_{j+1}]}\left\vert \frac{d^{k+1-p}}{dt^{k+1-p}}(\phi^{[p]})^{(i)}(\bar u(t))\right\vert\right),\quad \text{for all}~1\leq i\leq n.
\end{equation*}
Then~\eqref{eq:condition_Yinfty} holds.
\end{proposition}
\begin{proof}
We have $\Pi^n_{\infty}y=\Pi^n_{\infty}(G(\bar u)+\bar u)=\Pi^n_{\infty}g(\bar u)$. Since $\frac{d^p}{dt^p}g(\bar u)=\tau^p \phi^{[p]}(\bar u)$, we have that $\frac{d^{k+1}}{dt^{k+1}}g(\bar u)=\tau^p \frac{d^{k+1-p}}{dt^{k+1-p}}\phi^{[p]}(\bar u)$ and Proposition~\ref{prop:interpolation_error1} yields
\begin{align*}
\left\Vert\left(\Pi^n_{\infty}y\right)^{(i)} \right\Vert_{\infty} &\leq C_k \tau^p \max\limits_{0\leq j<m}\left((t_{j+1}-t_{j})^{k+1}\max\limits_{t\in[t_{j},t_{j+1}]}\left\vert \frac{d^{k+1-p}}{dt^{k+1-p}}(\phi^{[p]})^{(i)}(\bar u(t))\right\vert\right). \qedhere
\end{align*}
\end{proof}

\begin{remark} \label{rem:error_max}
{\em 
As comment similar to the one of Remark~\ref{rem:error_integrals} applies here. Indeed, the bound given in Proposition~\ref{prop:Y_infty} is not directly implementable because of the term
\begin{equation*}
\max\limits_{t\in[t_{j},t_{j+1}]}\left\vert \frac{d^{k+1-p}}{dt^{k+1-p}}(\phi^{[p]})^{(i)}(\bar u(t))\right\vert,
\end{equation*}
but we can again get an explicit bound for this quantity by using a low order Taylor approximation and interval arithmetic. In the particular case where the vector field $\phi$ is polynomial, an explicit bound can also be obtained \emph{via} the Chebyshev coefficients of the polynomial $\frac{d^{k+1-p}}{dt^{k+1-p}}(\phi^{[p]})^{(i)}(\bar u)$, as in~\eqref{eq:bound_infty_from_sum}.
}
\end{remark}

\subsection{The \boldmath$Z$\unboldmath~bounds}

In this section we derive the $Z$ bounds, which measure the contraction rate of $T$ on the ball of radius $r$ around $\bar u$. We begin with the finite dimensional part, that is the projection on $\X^n_{m,k}$. Let $z$ be defined as in~\eqref{eq:def_z}. Then

\begin{align*}
\Pi^n_{m,k} z =\ & \Pi^n_{m,k}\left(DT(\bar u+u_1)u_2\right)\\
=\ & \Pi^n_{m,k} u_2 -A_{m,k}\Pi^n_{m,k}\left(DG(\bar u+u_1)u_2\right) \\
=\ & \Pi^n_{m,k} u_2 -A_{m,k}D\Pi^n_{m,k}G(\bar u+u_1) u_2 \\
=\ & \left(Id-A_{m,k}A_{m,k}^{\dag}\right)\Pi^n_{m,k} u_2 -A_{m,k}\left(D\Pi^n_{m,k}G(\bar u+u_1)u_2-A_{m,k}^{\dag}\Pi^n_{m,k} u_2\right) \\
=\ & \left(Id-A_{m,k}A_{m,k}^{\dag}\right)\Pi^n_{m,k} u_2 -A_{m,k}\left(D\Pi^n_{m,k}G(\bar u)u_2-A_{m,k}^{\dag}\Pi^n_{m,k} u_2\right) \\
& -A_{m,k}\left(D\Pi^n_{m,k}G(\bar u+u_1)-D\Pi^n_{m,k}G(\bar u)\right)u_2,
\end{align*}
where $A_{m,k}$ and $A_{m,k}^{\dag}$ are defined as in Section~\ref{sec:operator}. We are going to bound each term separately as
\begin{align}
\nonumber
\left\vert \left(\Pi^n_{m,k} z\right)^{(i)}_{j,l}\right\vert \leq\ & \underbrace{\left\vert \left(\left(Id-A_{m,k}A_{m,k}^{\dag}\right)\Pi^n_{m,k} u_2\right)^{(i)}_{j,l}\right\vert}_{\leq \left(Z_0(r)\right)^{(i)}_{j,l}} + \underbrace{\left\vert \left(A_{m,k}\left(D\Pi^n_{m,k}G(\bar u)u_2-A_{m,k}^{\dag}\Pi^n_{m,k} u_2\right)\right)^{(i)}_{j,l}\right\vert}_{\leq \left(Z_1(r,r_{\infty})\right)^{(i)}_{j,l}} \\
& + \underbrace{\left\vert \left(A_{m,k}\left(D\Pi^n_{m,k}G(\bar u+u_1)-D\Pi^n_{m,k}G(\bar u)\right)u_2\right)^{(i)}_{j,l}\right\vert}_{\leq \left(Z_2(r,r_{\infty})\right)^{(i)}_{j,l}}.
\label{eq:Z_bounds_splitting}
\end{align}

\subsubsection{The bound \boldmath$Z_0(r)$\unboldmath}

The computation of the bounds $\left(Z_0(r)\right)^{(i)}_{j,l}$ estimating the first of the terms in the splitting \eqref{eq:Z_bounds_splitting} is rather straightforward and is simply a control on the precision of the numerical inverse.

\begin{proposition}
\label{prop:Z_0}
Let $u_2\in B_{\X^n}(r,r_{\infty})$, define the vector $\1_{\X^n_{m,k}}\in\R^{nm(k+1)}$ by $\left(\1_{\X^n_{m,k}}\right)^{(i)}_{j,l}=1$ for all $(i,j,l)\in\E^n_{m,k}$ and let
\begin{equation*}
\left(Z_0(r)\right)^{(i)}_{j,l} \bydef \left(\left\vert Id-A_{m,k}A_{m,k}^{\dag}\right\vert\1_{\X^n_{m,k}}\right)^{(i)}_{j,l}r,\quad \text{for all}~(i,j,l)\in\E^n_{m,k}.
\end{equation*}
Then, 
\begin{equation*}
\left\vert \left(\left(Id-A_{m,k}A_{m,k}^{\dag}\right)\Pi^n_{m,k} u_2\right)^{(i)}_{j,l}\right\vert \leq \left(Z_0(r)\right)^{(i)}_{j,l}, \quad \text{for all}~(i,j,l)\in\E^n_{m,k}.
\end{equation*}
\end{proposition}

\subsubsection{The bound \boldmath$Z_1(r,r_\infty)$\unboldmath}

We now construct the bounds $\left(Z_1(r,r_\infty)\right)^{(i)}_{j,l}$ estimating the second term in the splitting \eqref{eq:Z_bounds_splitting}. 
\begin{proposition}
\label{prop:Z_1}
Let $u_2\in B_{\X^n}(r,r_{\infty})$, consider $\rho=\left(\rho^{(i)}_{j,l}\right)_{(i,j,l)\in\E^n_{m,k}}$ such that
\begin{equation*}
\rho^{(i)}_{j,l} \geq r_{\infty}r\frac{\tau^p}{p!}(t_{j,l}-t_j)^{p}\max\limits_{s\in[t_j,t_{j+1}]}\left\vert D(\phi^{[p]})^{(i)}(\bar u(s))\right\vert \1_n, \quad \text{for all}~(i,j,l)\in\E^n_{m,k},
\end{equation*}
where $\1_{n}$ is the vector of size $n$ whose components all are equal to $1$. Let
\begin{equation*}
Z_1(r,r_{\infty})\bydef \left\vert A_{m,k}\right\vert \rho.
\end{equation*}
Then 
\begin{equation*}
\left\vert \left(A_{m,k}\left(D\Pi^n_{m,k}G(\bar u)u_2-A_{m,k}^{\dag}\Pi^n_{m,k} u_2\right)\right)^{(i)}_{j,l}\right\vert \leq \left(Z_1(r,r_{\infty})\right)^{(i)}_{j,l}, \quad \text{for all}~(i,j,l)\in\E^n_{m,k}.
\end{equation*}
\end{proposition}
\begin{proof}
By definition of $A_{m,k}^{\dag}$ and $\bar G$, we have that 
\begin{equation*}
D\Pi^n_{m,k}G(\bar u)\Pi^n_{m,k} u_2 = A_{m,k}^{\dag}\Pi^n_{m,k} u_2,\quad \text{for all}~u_2\in \X^n.
\end{equation*}
Therefore, we can rewrite
\begin{align*}
A_{m,k}\left(D\Pi^n_{m,k}G(\bar u)u_2-A_{m,k}^{\dag}\Pi^n_{m,k} u_2\right) &= A_{m,k} D\Pi^n_{m,k}G(\bar u)\left(u_2-\Pi^n_{m,k} u_2\right)\\
&= A_{m,k} D\Pi^n_{m,k}G(\bar u)\Pi^n_{\infty} u_2,
\end{align*}
and we only need to prove that
\begin{equation*}
\left\vert D\Pi^n_{m,k}G(\bar u)\Pi^n_{\infty} u_2\right\vert^{(i)}_{j,l} \leq \rho^{(i)}_{j,l}, \quad \text{for all}~(i,j,l)\in\E^n_{m,k}.
\end{equation*}
Remembering~\eqref{eq:projected_G} and using that $\Vert \Pi^n_{\infty} u_2\Vert_{\X^n_{\infty}}\leq r_{\infty}r$, we estimate for all $(i,j,l)\in\E^n_{m,k}$,
\begin{align*}
\left\vert D\Pi^n_{m,k}G(\bar u)\Pi^n_{\infty} u_2\right\vert^{(i)}_{j,l} &\leq 
r_{\infty}r\tau^p\int_{t_j}^{t_{j,l}}\frac{(t_{j,l}-s)^{p-1}}{(p-1)!}\left\vert D(\phi^{[p]})^{(i)}(\bar u(s))\right\vert \1_{n} ds \nonumber\\
&\leq 
r_{\infty}r\frac{\tau^p}{p!}(t_{j,l}-t_j)^{p}\max\limits_{s\in[t_j,t_{j+1}]}\left\vert D(\phi^{[p]})^{(i)}(\bar u(s))\right\vert \1_n.
\end{align*}
\end{proof}
Notice that Remark~\ref{rem:error_max} also applies here. 

\begin{remark} \label{rem:sharper_estimates}
{\em
Had we used the operator $\tilde g$ (see~\eqref{eq:tilde_g_original}) instead of $g$ (see~\eqref{eq:piecewise_g}), we would have gotten a bound like
\begin{align*}
\left\vert D\Pi^n_{m,k}G(\bar u)\Pi^n_{\infty} u_2\right\vert^{(i)}_{j,l} &\leq 
r_{\infty}r\tau^p\int_{0}^{t_{j,l}}\frac{(t_{j,l}-s)^{p-1}}{(p-1)!}\left\vert D(\phi^{[p]})^{(i)}(\bar u(s))\right\vert \1_{n} ds,
\end{align*}
which is obviously worst because one has to consider the whole integral from $0$ to $t_{j,l}$ instead of just from $t_j$ to $t_{j,l}$.
}
\end{remark}

\subsubsection{The bound \boldmath$Z_2$\unboldmath}

We finally construct the bounds $\left(Z_2(r,r_\infty)\right)^{(i)}_{j,l}$ estimating the last term in the splitting \eqref{eq:Z_bounds_splitting}.
\begin{proposition}
\label{prop:Z_2}
Let $u_1,u_2\in B_{\X^n}(r,r_{\infty})$. consider $\\varrho=\left(\\varrho^{(i)}_{j,l}\right)_{(i,j,l)\in\E^n_{m,k}}$ such that
\begin{align*}
\varrho^{(i)}_{j,l} \geq & \sum_{q=1}^{p-1}\frac{\tau^q}{q!}(t_{j,l}-t_j)^q \sum_{\delta=0}^{q(d-1)-1}\frac{1}{(1+\delta)!}\left\vert D^{2+\delta}(\phi^{[q]})^{(i)}(\bar u(t_j^-))\right\vert\left(\1_n^{2+\delta}\right)r^{2+\delta} \\
& \ +\frac{\tau^p}{p!}(t_{j,l}-t_j)^p \sum_{\delta=0}^{p(d-1)-1}\frac{1}{(1+\delta)!}\max\limits_{s\in[t_j,t_{j+1}]}\left\vert D^{2+\delta}(\phi^{[p]})^{(i)}(\bar u(s))\right\vert\left(\1_n^{2+\delta}\right)((\Lambda_k+r_{\infty})r)^{2+\delta}.
\end{align*}
Let
\begin{equation*}
Z_2(r,r_{\infty}) \bydef \left\vert A_{m,k}\right\vert \varrho.
\end{equation*}
Then 
\begin{equation*}
\left\vert \left(A_{m,k}\left(D\Pi^n_{m,k}G(\bar u+u_1)-D\Pi^n_{m,k}G(\bar u)\right)u_2\right)^{(i)}_{j,l}\right\vert \leq \left(Z_2(r,r_{\infty})\right)^{(i)}_{j,l}, \quad \text{for all}~(i,j,l)\in\E^n_{m,k}.
\end{equation*}
\end{proposition}
\begin{remark}
In the above proposition, $\left\vert D^{2+\delta}(\phi^{[p]})^{(i)}(\bar u(s))\right\vert\left(\1_n^{2+\delta}\right)$ must be understood as the evaluation of the $(2+\delta)$-linear form $\left\vert D^{2+\delta}(\phi^{[p]})^{(i)}(\bar u(s))\right\vert$ at the vectors $(\1_n,\ldots,\1_n)$, that is
\begin{align*}
\left\vert D^{2+\delta}(\phi^{[p]})^{(i)}(\bar u(s))\right\vert\left(\1_n^{2+\delta}\right) = \sum_{1\leq j_1,\ldots, j_{\delta+2}\leq n}\left\vert\partial_{j_1\ldots j_{\delta+2}}(\phi^{[p]})^{(i)}(\bar u(s))\right\vert .
\end{align*}
\end{remark}
\begin{proof} ({\em of Proposition~\ref{prop:Z_2}})
We only have to prove that 
\begin{equation*}
\left\vert \left( \left(D\Pi^n_{m,k}G(\bar u+u_1)-D\Pi^n_{m,k}G(\bar u)\right)u_2\right)^{(i)}_{j,l} \right\vert \leq \varrho^{(i)}_{j,l}, \quad \text{for all}~(i,j,l)\in\E^n_{m,k}.
\end{equation*}
Using~\eqref{eq:bound_infty_from_max} we have that 
\begin{equation*}
\Vert u_2^{(i)}\Vert_{\infty} \leq \Vert \Pi^1_{m,k} u_2^{(i)}\Vert_{\infty}+\Vert \Pi^1_{\infty} u_2^{(i)}\Vert_{\infty} \leq (\Lambda_k+r_{\infty})r.
\end{equation*} 
Then we estimate for all $(i,j,l)\in\E^n_{m,k}$,
\begin{align*}
&\left\vert \left(D\Pi^n_{m,k}G(\bar u+u_1)-D\Pi^n_{m,k}G(\bar u)\right)u_2\right\vert^{(i)}_{j,l}  \\ 
& \leq \sum_{q=1}^{p-1}\frac{\tau^q}{q!}(t_{j,l}-t_j)^q \left\vert \left(D(\phi^{[q]})^{(i)}(\bar u(t_j^-)+u_1(t_j^-)) - D(\phi^{[q]})^{(i)}(\bar u(t_j^-))\right)(u_2(t_j^-))\right\vert \\
& \quad +\tau^p\int_{t_j}^{t_{j,l}} \frac{(t_{j,l}-s)^{p-1}}{(p-1)!}\left\vert\left(D(\phi^{[p]})^{(i)}(\bar u(s)+u_1(s))-D(\phi^{[p]})^{(i)}(\bar u(s))\right)(u_2(s))\right\vert ds \\
& \leq \sum_{q=1}^{p-1}\frac{\tau^q}{q!}(t_{j,l}-t_j)^q \sum_{\delta=0}^{q(d-1)-1}\frac{1}{(1+\delta)!}\left\vert D^{2+\delta}(\phi^{[q]})^{(i)}(\bar u(t_j^-))\right\vert\left(\1_n^{2+\delta}\right)r^{2+\delta} \\
& \quad +\frac{\tau^p}{p!}(t_{j,l}-t_j)^p \sum_{\delta=0}^{p(d-1)-1}\frac{1}{(1+\delta)!}\max\limits_{s\in[t_j,t_{j+1}]}\left\vert D^{2+\delta}(\phi^{[p]})^{(i)}(\bar u(s))\right\vert\left(\1_n^{2+\delta}\right)((\Lambda_k+r_{\infty})r)^{2+\delta}. \qedhere
\end{align*}
\end{proof}
Notice that Remark~\ref{rem:error_max} also applies here.

\subsubsection{The \boldmath$Z_\infty$\unboldmath~bound}

We are left with the \emph{remainder part} of the $Z$ bound, which we treat in this section.

\begin{proposition}
\label{prop:Z_infty}
Let $u_1,u_2\in B_{\X^n}(r,r_{\infty})$ and $z$ as in~\eqref{eq:def_z}. 
Define for all $i\in\{1,\ldots,n\}$
\begin{align*}
&Z^{(i)}_{\infty}(r,r_{\infty})\geq\\
&\quad \tau^p C^{opt}_{k,p} \sum_{\delta=0}^{p(d-1)}\max\limits_{0\leq j<m}\left((t_{j+1}-t_{j})^p\frac{1}{\delta!}\max\limits_{t\in[t_{j},t_{j+1}]}\left\vert D^{1+\delta}(\phi^{[p]})^{(i)}(\bar u(t))\right\vert\left(\1_n^{1+\delta}\right)\right)\left((\Lambda_k+r_{\infty})r\right)^{1+\delta},
\end{align*}
where $C^{opt}_{k,p}$ is one of the two constants given by Propositions~\ref{prop:interpolation_error1} and \ref{prop:interpolation_error2}, namely
\begin{equation*}
C^{opt}_{k,p} = 
\begin{cases}
C_k &\text{if } p=k+1,\\
\tilde C_{k,p} &\text{if } p\leq k.
\end{cases}
\end{equation*}
Then~\eqref{eq:condition_Zinfty} holds.
\end{proposition}
\begin{proof}
We need to estimate
\begin{equation*}
\Pi^n_{\infty}z = \Pi^n_{\infty}\left(DT(\bar u+u_1)u_2\right) = \Pi^n_{\infty}\left(Dg(\bar u+u_1)u_2\right).
\end{equation*}
For any continuous function $\gamma$, one has
\begin{equation*}
\frac{d^{p}}{dt^{p}}\int_{t_j}^t \frac{(t-s)^{p-1}}{(p-1)!} \gamma(s) ds = \gamma(t),
\end{equation*}
thus we get, for all $1\leq i\leq n$
\begin{align*}
&\left\Vert \Pi^n_{\infty}\left(Dg^{(i)}(\bar u+u_1)u_2\right) \right\Vert_{\infty} \\
&\quad = \left\Vert \Pi^n_{\infty}\left(t\mapsto \tau^p\int_{t_j}^t \frac{(t-s)^{p-1}}{(p-1)!} D(\phi^{[p]})^{(i)}(\bar u(s)+u_1(s))u_2(s)ds\right) \right\Vert_{\infty}\\
&\quad \leq  \tau^p C^{opt}_{k,p} \max\limits_{0\leq j<m}\left((t_{j+1}-t_{j})^p\max\limits_{t\in[t_{j},t_{j+1}]}\left\vert D(\phi^{[p]})^{(i)}(\bar u(t)+u_1(t))u_2(t)\right\vert\right) \\
&\quad \leq  \tau^p C^{opt}_{k,p} \sum_{\delta=0}^{p(d-1)}\max\limits_{0\leq j<m}\left((t_{j+1}-t_{j})^p\frac{1}{\delta!}\max\limits_{t\in[t_{j},t_{j+1}]}\left\vert D^{1+\delta}(\phi^{[p]})^{(i)}(\bar u(t))\right\vert\left(\1_n^{1+\delta}\right)\right)\left((\Lambda_k+r_{\infty})r\right)^{1+\delta}. \qedhere
\end{align*}
\end{proof}
Notice that Remark~\ref{rem:error_max} also applies here.

\subsection{The radii polynomials and interval arithmetics}

The following proposition sums up what has been proven up to now in this section, namely that we have derived bounds that satisfy the requirements~\eqref{eq:condition_Y} to \eqref{eq:condition_Zinfty} from Theorem~\ref{th:rad_pol}.
\begin{proposition}
Let $y$ and $z$ defined as in~\eqref{eq:def_y} and \eqref{eq:def_z}. Then, the bound defined in Proposition~\ref{prop:Y_finite} satisfies~\eqref{eq:condition_Y} and the one from Proposition~\ref{prop:Y_infty} satisfies \eqref{eq:condition_Yinfty}. Also, consider the bounds defined in Propositions~\ref{prop:Z_0} to \ref{prop:Z_2}. Then
\begin{equation*}
Z^{(i)}_{j,l}(r,r_{\infty})=\left(Z_0(r)\right)^{(i)}_{j,l}+
\left(Z_1(r,r_{\infty})\right)^{(i)}_{j,l}+\left(Z_2(r,r_{\infty})\right)^{(i)}_{j,l},
\end{equation*}
satisfies~\eqref{eq:condition_Z} and finally the bound from Proposition~\ref{prop:Z_infty} satisfies~\eqref{eq:condition_Zinfty}.
\end{proposition}

Notice that, the way these bounds are defined, they are polynomials in $r$ and $r_{\infty}$, whose coefficients are all positive and can be computed explicitly with the help of the computer, since they depend on the numerical data of an approximate solution $\bar u$. Also, we make sure to control possible round-off errors by using interval arithmetic (in our case INTLAB~\cite{Ru99a}).

In practice, we first consider $r_{\infty}$ so that it satisfies the constraint~\eqref{eq:constraint_rinfty} introduced in the next section. If there does not exist such positive $r_{\infty}$, we increase $m$ and/or $k$ and/or $p$ and try again. Once $r_{\infty}$ is fixed, we try to find a positive $r$ such that the last conditions~\eqref{eq:condition_p_finite} and~\eqref{eq:condition_p_infty} of Theorem~\ref{th:rad_pol} hold. If there is no such positive $r$, we increase $m$ and/or $k$ and/or $p$ and try again. If we finally find a positive $r$ satisfying \eqref{eq:condition_p_finite} and~\eqref{eq:condition_p_infty}, then we have proven that Theorem~\ref{th:rad_pol} applies, that is there exists a unique zero of $G$ in $B_{\X^n}(r,r_{\infty})$. 

In Sections~\ref{sec:applications} and~\ref{sec:ABC}, we give several examples where the procedure described just above is successfully used to validate solutions of an initial value problem, as well as periodic solutions and heteroclinic orbits. But before doing so, we discuss in the next section the role of the parameters $k$, $m$ and $p$, and how they influence the bounds.

\section{About the choice of the parameters}
\label{sec:para}

In this section, we explain how the parameters $k$, $m$ and $p$ should be chosen, and in particular we highlight how the \emph{a priori bootstrap} (that is taking $p\geq 2$) helps improving the efficiency of the computer-assisted procedure we propose. The discussion will be rather informal, but we hope it helps the reader understand the results of the various comparisons presented in Section~\ref{sec:applications}. Also, to make things slightly simpler we assume here that the grid $\Delta_m$ is uniform, therefore in the estimates each instance of $t_{j+1}-t_j$ can be replaced by $\frac{1}{m}$.

Our main constraint is that we want the method to be successful while minimizing the size of our numerical data, that is the dimension of our finite dimensional space $\X^n_{m,k}$, which is $nm(k+1)$. Since $n$ is fixed by the dimension of the vector field $\phi$, we want to minimize the product $m(k+1)$. As we see in the examples of Section~\ref{sec:applications}, the usual limiting factor when trying to satisfy the radii polynomial inequalities \eqref{eq:condition_p_finite} and~\eqref{eq:condition_p_infty} is to get the order one term (in $r$) to be negative. For the finite part (that is \eqref{eq:condition_p_finite}), that means basically having that
\begin{equation*}
r_{\infty}\frac{\alpha}{p!}\left(\frac{\tau}{m}\right)^p<1,
\end{equation*}
(see Proposition~\ref{prop:Z_1}), and for the remainder part (that is \eqref{eq:condition_p_infty}) we get a condition like
\begin{equation*}
C^{opt}_{k,p}\left(\frac{\tau}{m}\right)^p\beta(\Lambda_k+r_{\infty})<r_{\infty},
\end{equation*}
where $\alpha$ and $\beta$ are constants depending on the numerical solution $\bar u$ and on the vector field $\phi$, but not on the parameters $k$, $m$ and $p$ that we can tune. This leads to
\begin{equation}
\label{eq:constraint_rinfty}
\beta C^{opt}_{k,p}\Lambda_k\left(\frac{\tau}{m}\right)^p < \left(1-\beta C^{opt}_{k,p}\left(\frac{\tau}{m}\right)^p\right)r_{\infty} < \left(1-\beta C^{opt}_{k,p}\left(\frac{\tau}{m}\right)^p\right)\frac{p!}{\alpha}\left(\frac{m}{\tau}\right)^p.
\end{equation}
We want to be able to chose a $r_{\infty}$ satisfying the above inequalities, and a necessary and sufficient condition for that is
\begin{equation*}
\beta C^{opt}_{k,p}\Lambda_k\left(\frac{\tau}{m}\right)^p < \left(1-\beta C^{opt}_{k,p}\left(\frac{\tau}{m}\right)^p\right)\frac{p!}{\alpha}\left(\frac{m}{\tau}\right)^p,
\end{equation*}
which we can rewrite 
\begin{equation}
\label{eq:necessary_condition}
\left(\frac{\tau}{m}\right)^p C^{opt}_{k,p} \left(\beta +\frac{\alpha\beta}{p!} \Lambda_k\left(\frac{\tau}{m}\right)^p\right) < 1.
\end{equation}
Remember that we want~\eqref{eq:necessary_condition} to be satisfied, while minimizing the product $m(k+1)$. When $p$ is fixed, and $k$ becomes large, notice that $C^{opt}_{k,p}$ is decreasing like $\frac{\ln(k)}{k^p}$. However, satisfying~\eqref{eq:necessary_condition} requires, roughly speaking, to decrease $\left(\frac{\tau}{m}\right)^p C^{opt}_{k,p}$ as much as possible. This suggests two things, which we confirm in our explicit examples of Section~\ref{sec:applications}. First, that it is slightly  better to increase $m$ than $k$ (because of the $ln(k)$ factor) and second, that if we take $p$ equal to 2 or more (that is if we use \emph{a priori bootstrap}) then we can satisfy~\eqref{eq:necessary_condition} while taking $m(k+1)$ much smaller than if we had $p$ equal to 1.

Finally, we point out that taking $k=p-1$ seems optimal for the conditon~\eqref{eq:necessary_condition} given by the order one term. Indeed, increasing $k$ from $p-1$ to $p$ increases the total number of coefficients, but brings no gain with respect to~\eqref{eq:necessary_condition} since
\begin{equation*}
C^{opt}_{p-1,p}=C_{p-1}<\tC_{p,p}=C^{opt}_{p,p}.
\end{equation*}
However, for the proof to succeed (that is for \eqref{eq:condition_p_finite} and \eqref{eq:condition_p_infty} to be satisfied) we also need small enough $Y$ and $Y_{\infty}$ bounds. Looking more precisely at $Y_{\infty}$, we see that it is of the form
\begin{equation*}
C_k \frac{1}{m^{k+1}}\gamma,
\end{equation*}
where $\gamma$ depends on the numerical data $\bar u$ and also on $k$, but the dependency on $k$ is way less important than in the $C_k \frac{1}{m^{k+1}}$ term, so we neglect it here. Looking back to the definition of $C_k$ in Proposition~\ref{prop:interpolation_error1}, we see that the term that we want to be small is of the form
\begin{equation*}
\frac{1}{(k+1)!}\frac{4}{(4m)^{k+1}}.
\end{equation*} 
Therefore, if we really need to decrease the $Y_{\infty}$ bound, increasing $k$ is drastically better than increasing $m$. That is why, in practice we often take $k=p$, even though $k=p-1$ would be enough to satisfy~\eqref{eq:necessary_condition}. Finally we point out that, if we are not simply focused on getting an existence result, but also care about having sharp error bound, then we should definitively take care of having small $Y$ and $Y_{\infty}$ bounds, which, as we will show in the next section, can be achieved by slightly increasing $k$ (that is taking $k>p$).

In the next section, we present several comparisons for different choices of parameters, that confirm the heuristic presented in this section.

\section{Examples of applications for the Lorenz system}
\label{sec:applications}

In this section, we consider the Lorenz system, that is
\begin{equation*}
\phi(x,y,z)=\begin{pmatrix}
\sigma(y-x)\\
\rho x -y -xz\\
-\beta z +xy
\end{pmatrix},
\end{equation*}
with standard parameter values $(\sigma,\beta,\rho)=(10,\frac{8}{3},28)$. 
Here, we first consider the initial value problem~ \eqref{eq:general_IVP}, and use those bounds to try and validate orbits of various length with different parameters, to highlight the significant improvement made possible by the \emph{a priori bootstrap} technique (that is taking $p\geq 2$). Then, we show that the \emph{a priori bootstrap} also allows to validate more interesting solutions (from a dynamical point of view), namely periodic orbits and connecting orbits.

\subsection{Comparisons for the initial value problem}
\label{sec:IVP}

The aim of this section is to showcase the improvements allowed by the use of \emph{a priori bootstrap}, and to validate the heuristics made in Section~\ref{sec:para}. To do so, we fix an initial data (chosen close to the attractor of the Lorenz system)
\begin{equation}
\label{eq:initial_value}
u_0=\begin{pmatrix}
-14.68 \\ -11 \\ 37.67
\end{pmatrix},
\end{equation}
and do two kinds of comparisons. First, we try to validate the longest possible orbits for $p=1,2,3$ at various values of $m$ and $k$. We recall that by validating, we mean getting the existence of a true solution near a numerical one, by checking that the hypotheses of Theorem~\ref{th:rad_pol} hold. To make the comparison fair, we fix the total number of coefficients used for the numerical approximation, that is the dimension of $\X^n_{m,k}$, given by $nm(k+1)$. This quantity is usually the bottleneck of our approach, since we need to store and invert the matrix $A^{\dag}_{m,k}$ which is of size $nm(k+1)\times nm(k+1)$. Here, we take $nm(k+1)=14000$ (or as close as possible to $14000$). The computations were made on a laptop with 8GB of RAM, and of course $nm(k+1)$ could be taken larger on a computer with more memory.

The first set of results are given in Table~\ref{table:p1} (we recall that we work with the Lorenz system, therefore $n=3$).

\begin{table}[h!]
\begin{center}
\begin{tabular}{|c|c|c|c|c|}
  \hline
  $k$ & $m$ & $nm(k+1)$ & $\tau_{max}$ & $r$ \\
  \hline
  1 & 2333 & 13998 & \textcolor{red}{0.69} & $2.3233\times 10^{-5}$ \\
  2 & 1556 & 14004 & 0.64 & $1.1524\times 10^{-7}$ \\
  3 & 1167 & 14004 & 0.58 & $8.6805\times 10^{-9}$ \\
  \hline
\end{tabular}
\caption{Comparisons for $p=1$. $\tau_{max}$ is the longest integration time for which the proof succeeds, and $r$ is the associated validation radius, that is a bound of the distance (in $\CC^0$ norm) between the numerical data used and the true solution.}
\end{center}
\label{table:p1}
\end{table}
In all cases, the proof fails for longer time $\tau$, because~\eqref{eq:necessary_condition} is no longer satisfied. We see here that, as announced in Section~\ref{sec:para}, it is better to take $k$ as small as possible to get the longest possible orbit, but that increasing $k$ helps reducing the $Y_{\infty}$ bound, and thus the validation radius $r$. We see that simply increasing the order of the polynomial interpolation (given by $k$), allows to get better accuracy but does not really help to prove longer orbits. However, we are going to show on the next examples (see Table~\ref{table:p2}) that combining \emph{a priori bootstrap} (that is taking $p\geq 2$) with higher order polynomial interpolation does allow to get much longer orbits.

\begin{table}[h!]
\begin{center}
\begin{tabular}{|c|c|c|c|c|}
  \hline
  $k$ & $m$ & $nm(k+1)$ & $\tau_{max}$ & $r$ \\
  \hline
  1 & 2333 & 13998 & 0.97 & $1.5718\times 10^{-3}$ \\
  2 & 1556 & 14004 & \textcolor{red}{5.6} & $8.4373\times 10^{-5}$ \\
  3 & 1167 & 14004 & 5.5 & $8.4184\times 10^{-8}$ \\
  4 & 933  & 13995 & 4.9 & $7.9190\times 10^{-9}$ \\
  \hline
\end{tabular}
\caption{Comparisons for $p=2$. $\tau_{max}$ is the longest integration time for which the proof succeeds, and $r$ is the associated validation radius, that is a bound of the distance (in $\CC^0$ norm) between the numerical data used and the true solution.}
\label{table:p2}
\end{center}
\end{table}
First, comparing the $k=1$ case when $p=1$ and $p=2$, we see that using \emph{a priori bootstrap} allows to get a slightly longer orbit, even for linear interpolation. Also, even for the longest possible orbit in that case ($\tau=0.97$), we still have much room to satisfy~\eqref{eq:necessary_condition} (the quantity given by~\eqref{eq:necessary_condition} is $\ll 1$). However, we cannot get a longer orbit in that case even with $p=2$, because the $Y_{\infty}$ bound becomes too large. This can be dealt with by increasing $k$, and we see that we can then get much longer orbits. To finish this set of comparisons, we show that doing one more iteration of the \emph{a priori bootstrap} process (that is taking $p=3$ instead of $p=2$) still improves the results and allows to get longer orbits (see Table~\ref{table:p3}).

\begin{table}[h!]
\begin{center}
\begin{tabular}{|c|c|c|c|c|}
  \hline
  $k$ & $m$ & $nm(k+1)$ & $\tau_{max}$ & $r$ \\
  \hline
  2 & 1556 & 14004 & 5.6 & $7.6716\times 10^{-4}$ \\
  3 & 1167 & 14004 & \textcolor{red}{8.1} & $9.3043\times 10^{-6}$ \\
  4 & 933  & 13995 & \textcolor{red}{8.1} & $8.8204\times 10^{-8}$ \\
  5 & 778  & 14004 & 8.0 & $1.6175\times 10^{-8}$ \\
  9 & 467  & 14010 & 7.9 & $1.3748\times 10^{-8}$ \\
  19 & 233  & 13980 & 6.9 & $2.2998\times 10^{-8}$ \\
  \hline
\end{tabular}
\caption{Comparisons for $p=3$. $\tau_{max}$ is the longest integration time for which the proof succeeds, and $r$ is the associated validation radius, that is a bound of the distance (in $\CC^0$ norm) between the numerical data used and the true solution.}
\end{center}
\label{table:p3}
\end{table}

We sum up this set of comparisons by displaying the longest orbit obtained with $p=1$, $p=2$ and $p=3$ (see Figure~\ref{fig:longest_orbit}).

\begin{figure} [h!]
\begin{center}
\subfigure{\includegraphics[width=7.5cm]{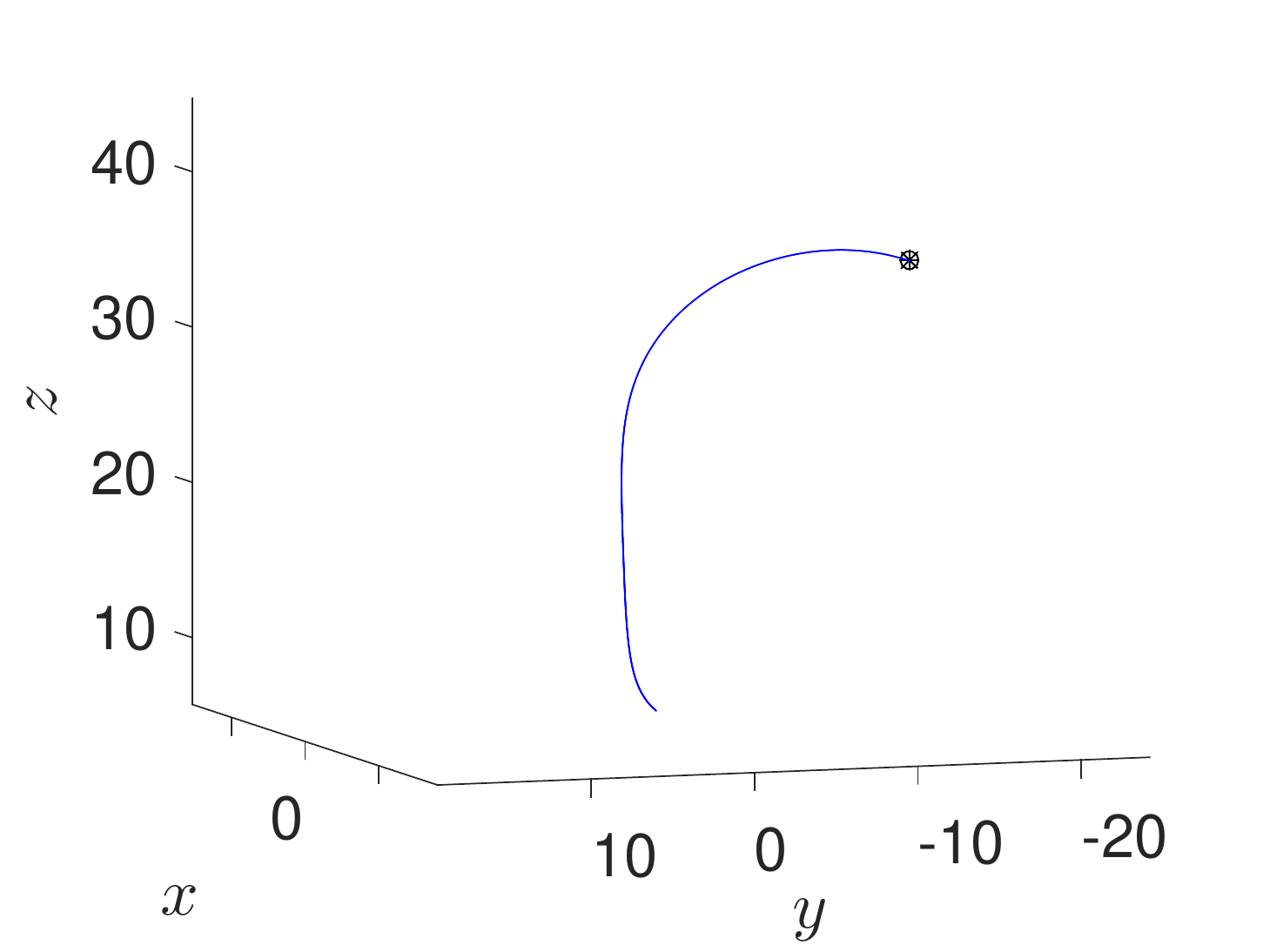}}
\subfigure{\includegraphics[width=7.5cm]{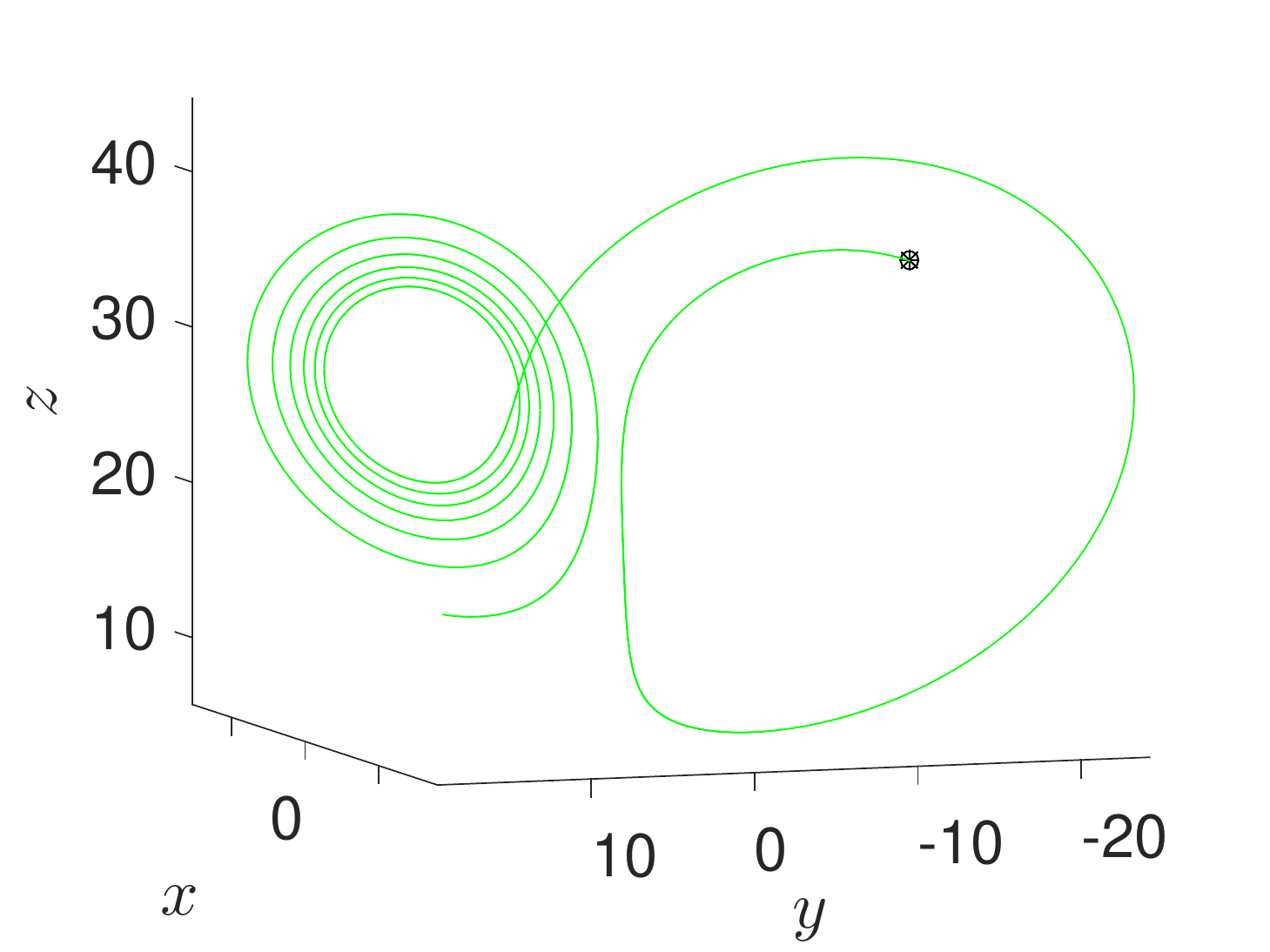}} 
\subfigure{\includegraphics[width=7.5cm]{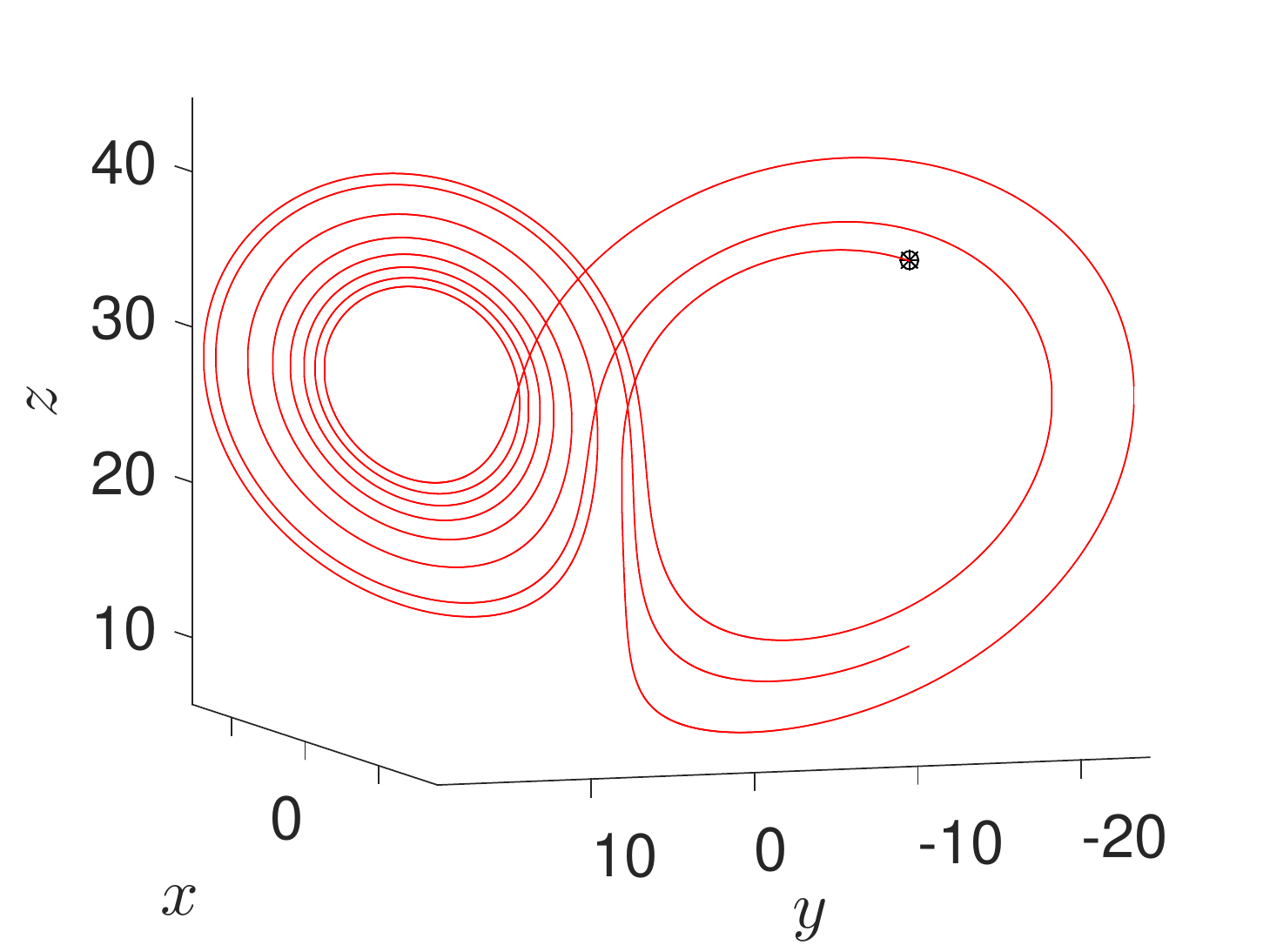}} 
\end{center}
\vspace{-.4cm}
\caption{The longest orbits we are able to validate, with a total number of coefficient of approximately 14000. In blue for $p=1$, in green for $p=2$ and in red for $p=3$. The initial value is given by~\eqref{eq:initial_value}.} 
\label{fig:longest_orbit}
\end{figure}

We then finish this section with another set of comparisons, where we now fix the length of the orbit, here $\tau=2$, and instead look for the minimal total number of coefficients for which we can validate this orbit (for different values of $p$). The aim of this experiment is to show that using \emph{a priori bootstrap} enables to validate solutions that one would not be able to validate without using it. Indeed, we are going to see that taking $p$ greater than one allows us to use way less coefficients to validate the solutions. Thus, if for a given solution, the proof without \emph{a priori bootstrap} requires more coefficients than what our computer can handle, one can reduce this number by using \emph{a priori bootstrap} and then possibly validate the orbit. For instance, still with the initial condition given by~\eqref{eq:initial_value}, we cannot validate the orbit of length $\tau=2$ without \emph{a priori bootstrap} (that is with $p=1$), at least not with less that $14000$ coefficients. However, the next table of results shows that we can validate it with $p=2$, and also using even less coefficients with $p=3$.

\begin{table}[H]
\begin{center}
\begin{tabular}{|c|c|c|c|c|}
  \hline
        & $k=1$    & $k=2$          & $k=3$          & $k=4$          \\
  $p=2$ & no proof & $m=416$        & $m=415$        & $m=377$        \\
        & no proof & \textcolor{red}{$nm(k+1)=3744$} & $nm(k+1)=4980$ & $nm(k+1)=5655$ \\
  \hline
        & $k=2$          & $k=3$          & $k=4$          & $k=5$          \\
  $p=3$ & $m=470$        & $m=125$        & $m=110$        & $m=99$        \\
        & $nm(k+1)=4230$ & \textcolor{red}{$nm(k+1)=1500$} & $nm(k+1)=1650$ & $nm(k+1)=1782$ \\
  \hline
\end{tabular}
\caption{Minimal number of coefficients needed to validate the orbit of length $\tau=2$, starting from $u_0$ given in~\eqref{eq:initial_value}.}
\end{center}
\end{table}

\subsection{Validation of a periodic orbit.}
\label{sec:periodic}

To study periodic orbits, instead of an initial value problem, the system~\eqref{eq:general_problem} has to be slightly modified into a boundary value problem
\begin{equation}
\label{eq:periodic_problem}
\left\{
\begin{aligned}
&u'(t)=\phi(u(t)), \quad t\in[0,\tau],\\
&u(0)=u(\tau), \\
&\left\langle u(0)-u_0,v_0\right\rangle = 0,
\end{aligned}
\right.
\end{equation}
where $\tau$ is now an unknown of the problem, and where $u_0,v_0 \in \mathbb R^n$. The last equation is sometimes called \emph{Poincaré phase condition} and is here to isolate to periodic orbit.

As for the initial value problem, we can then consider an equivalent integral formulation (possibly with \emph{a priori bootstrap}) and define an equivalent fixed point operator $T$ very similar to the one introduced in Section~\ref{sec:general}. The additional phase condition and the fact that $\tau$ is now a variable only require minor modifications of $T$ and of the bounds derived in Section~\ref{sec:bounds} (see for intance \cite{BerShe15}).

Using \emph{a priori bootsrtap}, we are able to validate fairly complicated periodic orbits (see Figure~\ref{fig:periodic}).

\begin{figure}[h!] 
\begin{center}
\includegraphics[width=10cm]{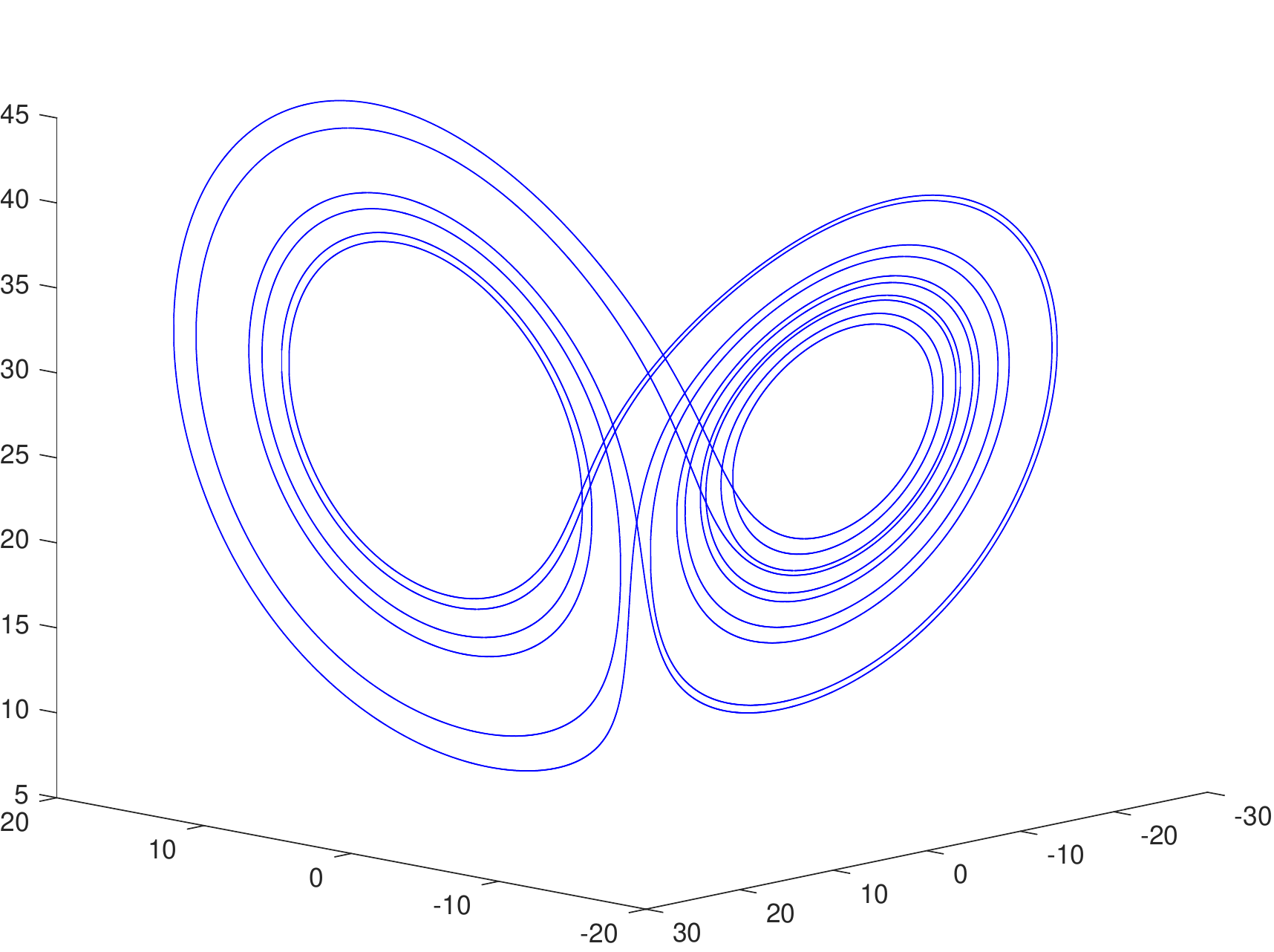}
\end{center}
\vspace{-.3cm}
\caption{A validated periodic orbit of the Lorenz system, whose period $\tau$ is approximately $11.9973$. We used two iterations of \emph{a priori bootstrap}, that is $p=3$, for the validation. If we want to minimize the total number of coefficients to do the validation, we can take $k=3$ and $m=602$ (which makes $7225$ coefficients in total), and we then get a validation radius of $1.5627\times 10^{-4}$. It is possible to get a significantly lower validation radius, at the expense of a slight increase in the total number of coefficients: for instance with $k=5$ and $m=495$ (which makes $8911$ coefficients in total), we get a validation radius of $4.7936\times 10^{-9}$.} 
\label{fig:periodic}
\end{figure}

\subsection{Validation of a connecting orbit}
\label{sec:connecting}

In this section, we present a computer-assisted proof of existence of a connecting orbit in the Lorenz system for the standard parameter values $(\sigma,\beta,\rho)=(10,\frac{8}{3},28)$. It is well know that at these parameter values, the Lorenz system admits a transverse connecting orbit between $\left(\sqrt{\beta(\rho-1)},\sqrt{\beta(\rho-1)},\rho-1\right)$ and the origin. 

While computer-assisted proofs of connecting orbits were already investigated several times using topological and analytical approaches (see~\cite{MR2821596,MR3207723,MR1961956,MR1236201,MR2157844,MR2302059,MR2339601,MR945967,MR1726672,MR2060531,MR2494688,MR2505658,MR2173545,MR1661847,MR2388394}), a paper of particular relevance to the present work is~\cite{MR3207723}, where a particular case of our method is developed, with only linear interpolation and no \emph{a priori bootstrap} (that is $k=1$ and $p=1$). While the authors in \cite{MR3207723} were able to validate several connecting orbits for the Lorenz system, they could not validate the aforementioned connecting orbit for the standard parameter values. In fact, one of the main motivations for the present work was to improve the setting of~\cite{MR3207723} to be able to validate more complicated orbits.

As we showcased in Section~\ref{sec:IVP}, using \emph{a priori bootstrap} enables us to validate significantly more complicated orbits for the initial value problem, and this is also true for connecting orbits. Indeed we are able to validate the standard connecting orbit for the Lorenz system, with parameter values $(\sigma,\beta,\rho)=(10,\frac{8}{3},28)$. Before exposing the results, we briefly describe how to modify~\eqref{eq:general_IVP} to be able to handle connecting orbits.

Compared to an initial value problem on a given time interval, or to a periodic orbit, connecting orbits present an aditionnal difficulty which is that they are defined on an infinite time interval (from $-\infty$ to $+\infty$). To circumvent this difficulty and get back to a time interval of finite length, which is more suited to numerical computations (and to computer-assisted proofs), we are going to use local stable and unstable manifolds of the fixed points. By a computer-assisted method very similar to the one presented here, we first compute and validate local parameterization of the unstable manifold at $\left(\sqrt{\beta(\rho-1)},\sqrt{\beta(\rho-1)},\rho-1\right)$ and of the stable manifold at the origin. Since the main object of this work is not the rigorous computations of those manifolds, we simply assume that they are given (with validation radius) and do not give more details here. The interested reader can find more information about the computations and validations of these parameterizations in~\cite{MR3437754,MR3518609} and the references therein, and also more detailed examples of their usage to get connecting orbits in~\cite{MR2821596,MR3207723,MR3353132,BerShe15}.

We denote by $P$ a local parameterization of the stable manifold of the origin, and by $Q$  a local parameterization of the unstable manifold of  a local parameterization of the stable manifold of $\left(\sqrt{\beta(\rho-1)},\sqrt{\beta(\rho-1)},\rho-1\right)$. We point out that both manifolds are two-dimensional. We then want to solve

\begin{equation}
\label{eq:connecting_problem}
\left\{
\begin{aligned}
&u'(t)=\phi(u(t)), \quad t\in[0,\tau],\\
&u(0)=Q(\varphi), \\
&u(\tau)=P(\theta),
\end{aligned}
\right.
\end{equation}
where $\varphi$ and $\theta$ each is a one dimensional parameter, the parameter in the other dimension being fixed to isolate the solution. Notice that $\tau$ is now an unknown of the system. As for the initial value problem, we can then consider an equivalent integral formulation (possibly with \emph{a priori bootstrap}) and define an equivalent fixed point operator $T$ very similar to the one introduced in Section~\ref{sec:general}. The additional equation $u(\tau)=P(\theta)$ and the fact that we have three additional variables $\tau$, $\theta$ and $\varphi$ only requires minor modifications of $T$ and of the bounds derived in Section~\ref{sec:bounds} (see for intance \cite{BerShe15,MR3207723}).

Using $p=3$, $k=3$ and $m=1150$ (that is a total number of 13803 coefficients) we are then able to rigorously compute a solution of~\eqref{eq:connecting_problem} (see Figure~\ref{fig:connecting_orbit}).

\begin{figure}[h!] 
\begin{center}
\includegraphics[width=12cm]{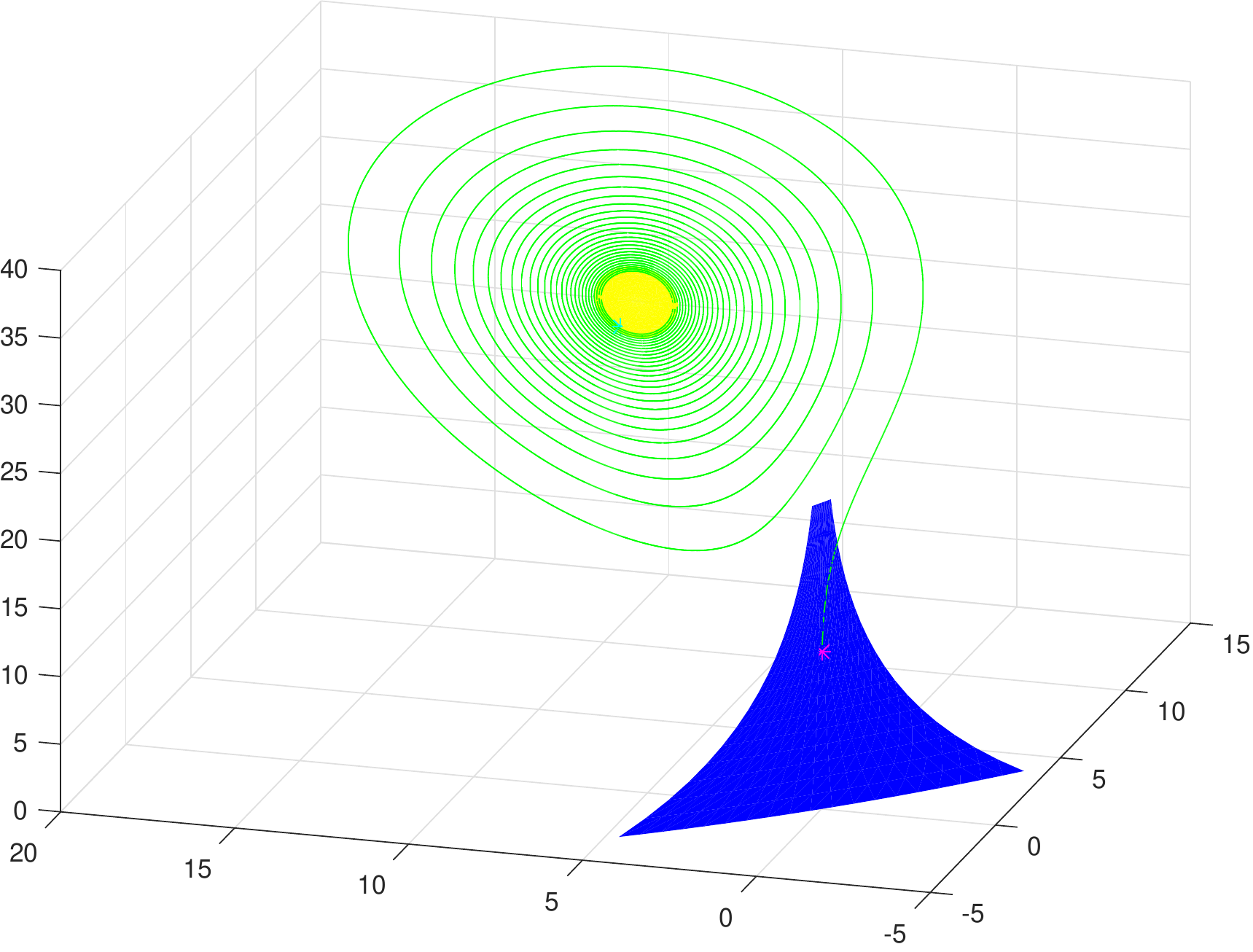}
\end{center}
\vspace{-.3cm}
\caption{Validated connecting orbit for the Lorenz system, with parameters $(\sigma,\beta,\rho)=(10,\frac{8}{3},28)$. The local stable manifold of the origin is in blue, the local unstable manifold of $\left(\sqrt{\beta(\rho-1)},\sqrt{\beta(\rho-1)},\rho-1\right)$ in yellow, and the green connection between them (of length $\tau\simeq 17.3$) is validated using polynomial interpolation, with \emph{a priori bootstrap} ($p=3$). The proof gives a validation radius of $r=3.1340\times 10^{-5}$.} 
\label{fig:connecting_orbit}
\end{figure}

\section{Examples of applications for ABC flows}
\label{sec:ABC}

In this section, we apply our method to the non polynomial vector field 
\begin{equation*}
\phi_{A,B,C}(x,y,z) \bydef \begin{pmatrix}
A\sin(z) + C\cos(y) \\ B\sin(x) + A\cos(z) \\ C\sin(y) + B\cos(x)
\end{pmatrix}, \quad A,B,C \in \R.
\end{equation*}
The map $\phi_{A,B,C}$ is usually referred to as the Arnold-Beltrami-Childress (ABC) vector field, and gives a prime example of complex steady incompressible periodic flows in 3D (see~\cite{Arnold_ABC,MR851673,ref2_ABC} and the references therein).

The main point of this section is to briefly illustrate the applicability of our technique to non polynomial vector fields. We plan on studying more thoroughly ABC flows with the help of our a posteriori validation method in a future work. Recently, the existence of orbits, that are periodic up to a shift of $2\pi$ in one coordinates, have been proven in the cases $A=B=C=1$ and $0<A\ll 1,\ B=C=1$ \cite{MR3549020,MR3580814}. Applying the method developed in this paper, we were able to complete these results by proving the following statements.
\begin{theorem}
\label{th:variable_A}
For all $A=0.1,0.2,\ldots,1$, with $B=C=1$, there exists $\tau_A\in[\tau^-_A,\tau^+_A]$ (see Table~\ref{table:tau_A}) and a solution $(x,y,z)$ of the ABC flow such that
\begin{equation*}
x(t+\tau)=x(t)+2\pi,\quad y(t+\tau)=y(t),\quad z(t+\tau)=z(t),\quad\forall~t\in\R.
\end{equation*}
\end{theorem}
\begin{proof}
The proof is done by running \verb+script_proofs_A11.m+ (available at~\cite{webpage_AprioriBootstrap}), which for each $A=0.1,0.2,\ldots,1$ computes an approximate solution, then computes bounds satisfying~\eqref{eq:condition_Y}-\eqref{eq:condition_Zinfty} as described in Section~\ref{sec:bounds}, and finally finds positive $r_{\infty}$ and $r$ such that~\eqref{eq:condition_p_finite}-\eqref{eq:condition_p_infty} holds. 
\end{proof}
\begin{theorem}
\label{th:4pi}
For all $A=B=C=1$, there exists $\tau\in[7.797656,7.797666]$ and a solution $(x,y,z)$ of the ABC flow such that
\begin{equation*}
x(t+\tau)=x(t)+4\pi,\quad y(t+\tau)=y(t),\quad z(t+\tau)=z(t),\quad\forall~t\in\R
\end{equation*}
and $x(\cdot+\tau)\neq x(\cdot)+2\pi$.
\end{theorem}
\begin{proof}
The proof is done by running \verb+script_proofs_111.m+ (available at~\cite{webpage_AprioriBootstrap}), which computes an approximate solution, then computes bounds satisfying~\eqref{eq:condition_Y}-\eqref{eq:condition_Zinfty} as described in Section~\ref{sec:bounds}, and finally finds positive $r_{\infty}$ and $r$ such that~\eqref{eq:condition_p_finite}-\eqref{eq:condition_p_infty} holds. 
\end{proof}
The solutions given by Theorem~\ref{th:variable_A} are represented in Figure~\ref{fig:variable_A}, and the solution given by Theorem~\ref{th:4pi} is represented in Figure~\ref{fig:4pi}.

\begin{table}[htbp]
\centering
\begin{tabular}{|c||c|c|c|c|c|}
\hline
$A$ & $1$ & $0.9$ & $0.8$ & $0.7$ & $0.6$ \\
\hline
$\tau^-_A$ & $3.23527736$ & $3.41779635$ & $3.62512508$ & $3.86405419$ & $4.14464726$ \\
\hline
$\tau^+_A$ & $3.23527746$ & $3.41779647$ & $3.62512521$ & $3.86405436$ & $4.14464749$ \\
\hline
\end{tabular}
\begin{tabular}{|c||c|c|c|c|c|}
\hline
$A$ & $0.5$ & $0.4$ & $0.3$ & $0.2$ & $0.1$ \\
\hline
$\tau^-_A$ & $4.48269179$ & $4.90491344$ & $5.46177978$ & $6.26680147$ & $7.67945129$ \\
\hline
$\tau^+_A$ & $4.48269213$ & $4.90491401$ & $5.46178092$ & $6.26680442$ & $7.67946552$ \\
\hline
\end{tabular}
\caption{The intervals $[\tau^-_A,\tau^+_A]$, for $A=0.1,0.2,\ldots,1$ in which the \emph{period} $\tau_A$ of the solution described in Theorem~\ref{th:variable_A} is proved to be.}
\label{table:tau_A}
\end{table}

\begin{figure} [htbp]
\begin{center}
\subfigure{\includegraphics[width=7.5cm]{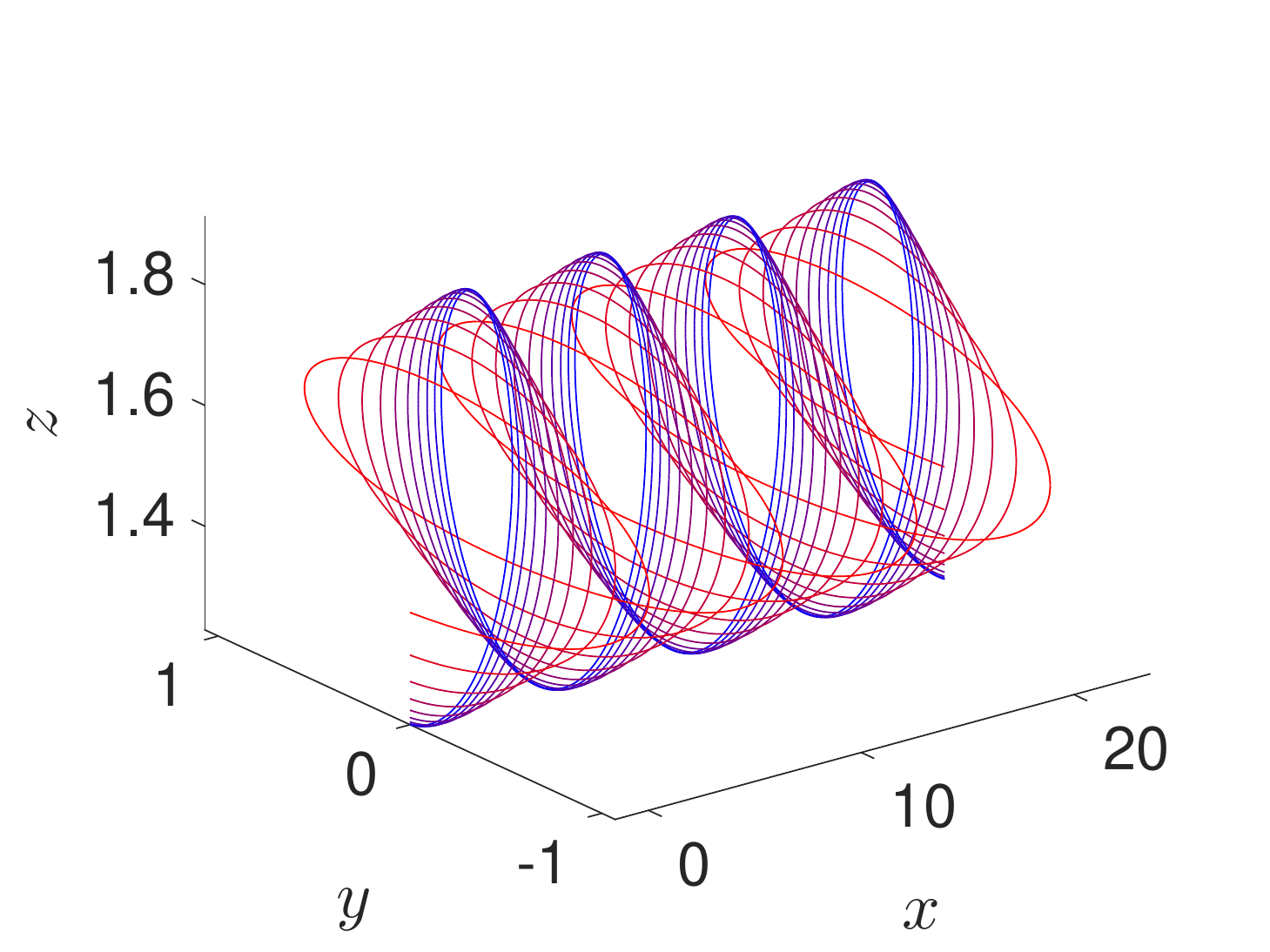}}
\subfigure{\includegraphics[width=7.5cm]{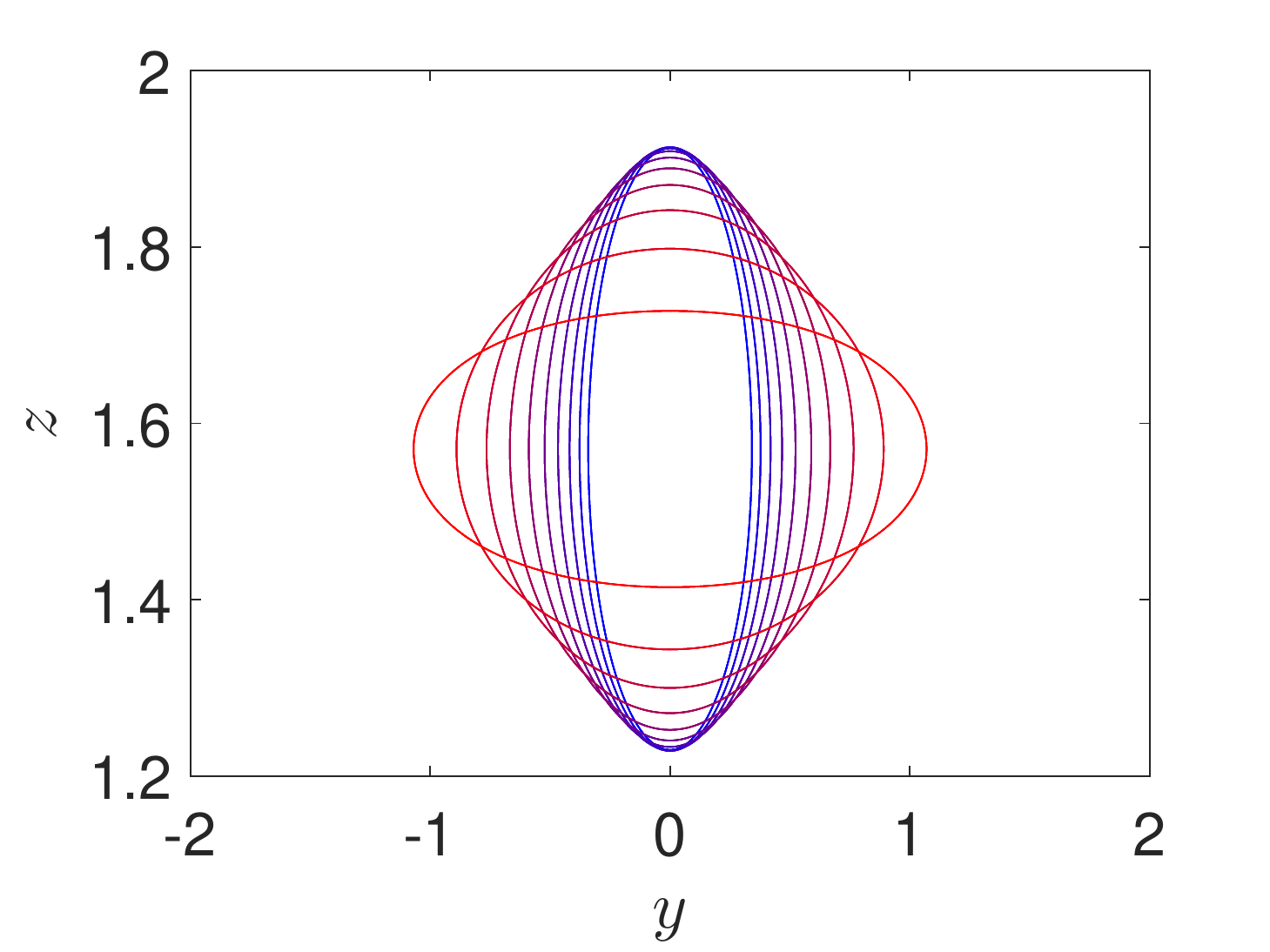}} 
\end{center}
\vspace{-.4cm}
\caption{These are the orbits that are described in Theorem~\ref{th:variable_A}. The color varies from blue for $A=1$ to red for $A=0.1$. Each proof was done with $p=2$, $k=2$ and $m=50$, and gave a validation radius varying from $r=4.8313\times 10^{-8}$ to $r=7.4012\times 10^{-6}$.}
\label{fig:variable_A}
\end{figure}

\begin{figure} [htbp]
\begin{center}
\subfigure{\includegraphics[width=7.5cm]{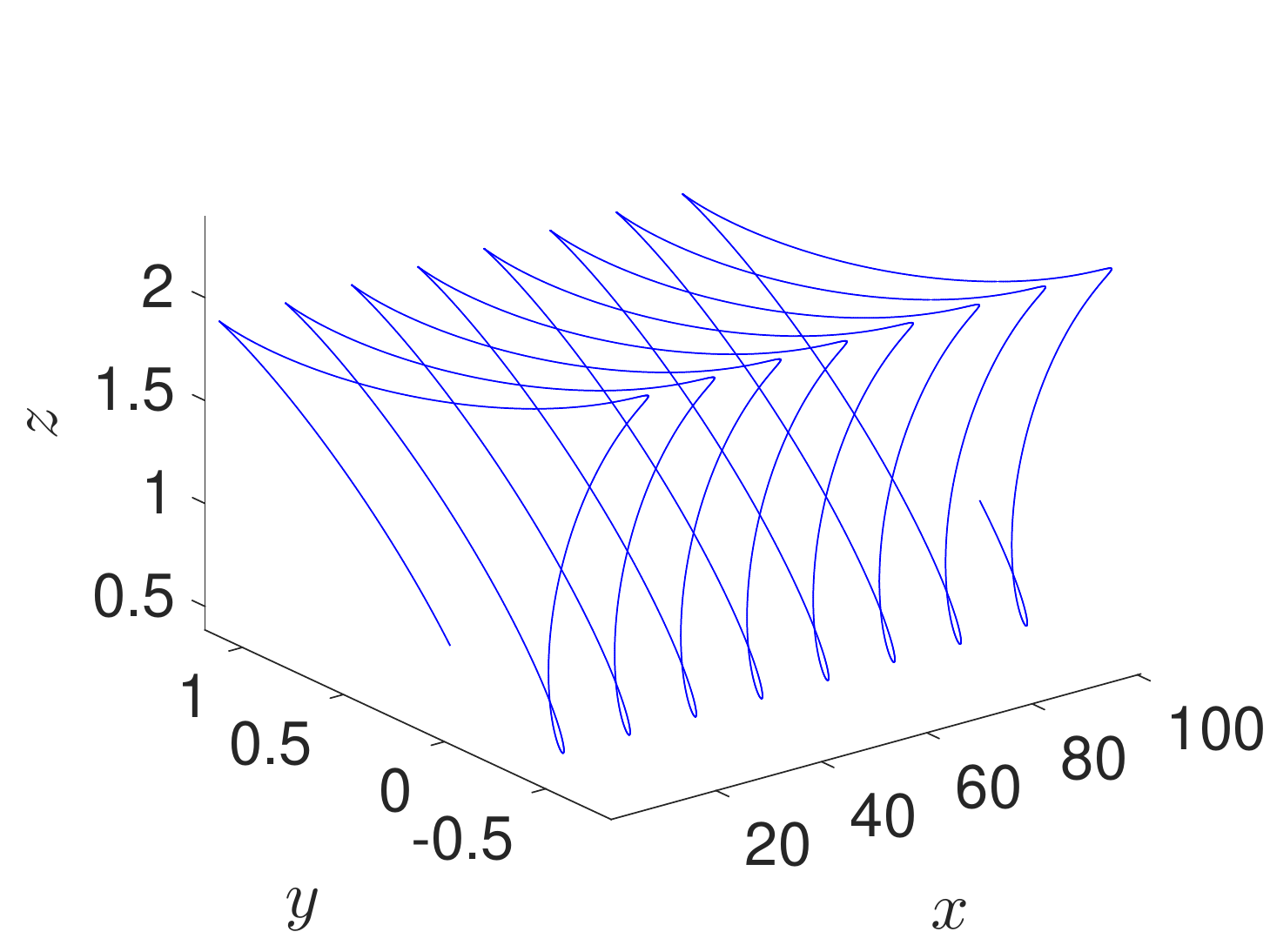}}
\subfigure{\includegraphics[width=7.5cm]{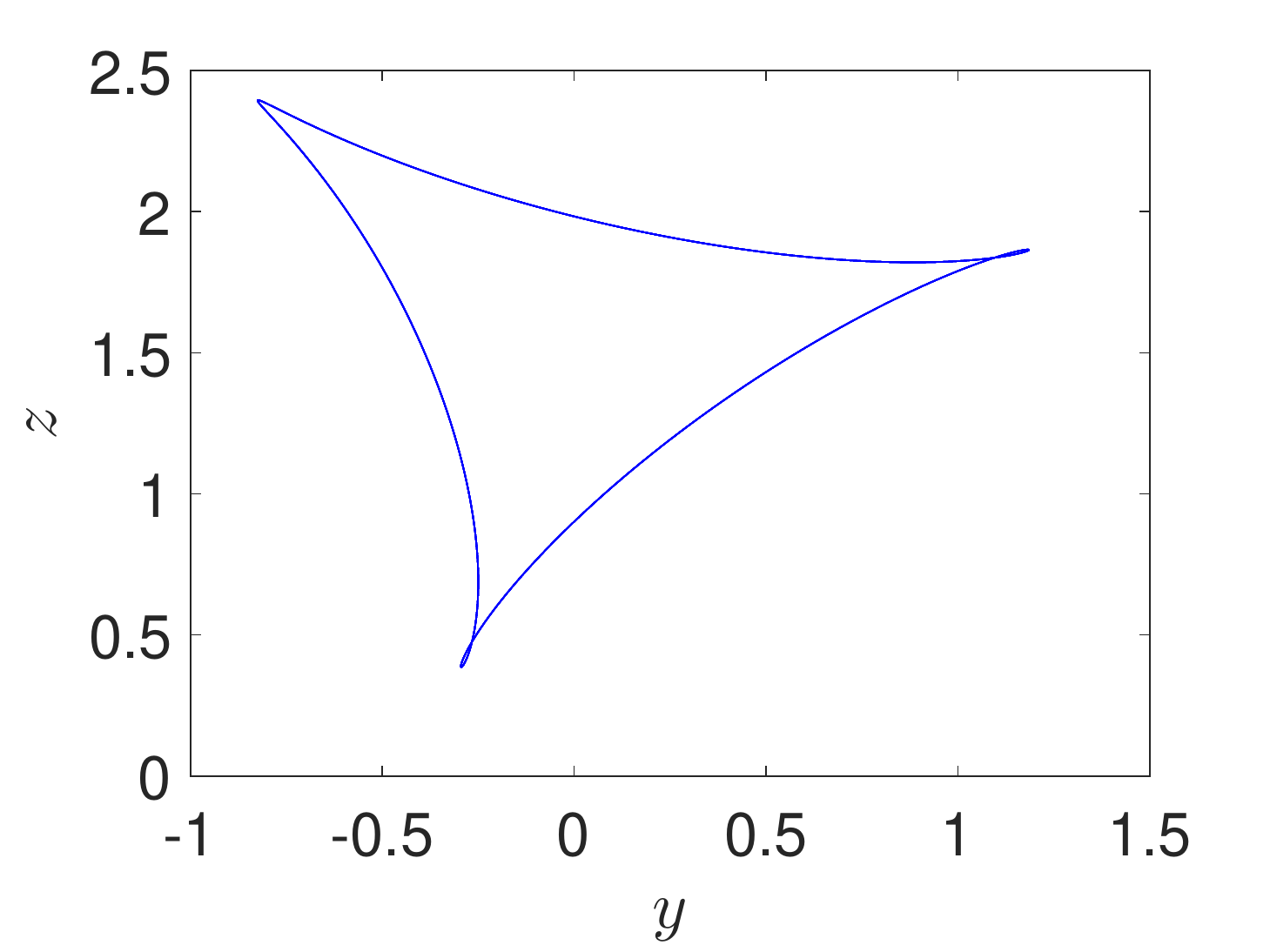}} 
\end{center}
\vspace{-.4cm}
\caption{This is the orbit that is described in Theorem~\ref{th:4pi}. The proof was done with $p=2$, $k=2$ and $m=300$, and gave a validation radius $r=4.0458\times 10^{-6}$.}
\label{fig:4pi}
\end{figure}

\section*{Appendix}

For the sake of completeness, we give here some properties of the Lebesgue constant $\Lambda_k$, as well as proofs of Proposition~\ref{prop:interpolation_error1} and Proposition~\ref{prop:interpolation_error2}. We will assume here that $u$ is defined and smooth on $[-1,1]$. The corresponding estimates on $[t_j,t_{j+1}]$ can the easily be deduced by rescaling.

We recall that $\Lambda_k$ is defined as the norm of the interpolation operator mapping $\CC^0([-1,1],\R)$ to itself, and associating a continuous function $u$ to its interpolation polynomial $P_k(u)$ of order $k$. Of course this operator (and its norm) depend on the interpolation points, which in this work are the Chebyshev interpolation points of the second kind
\begin{equation*}
x_l^k=\cos\left(\frac{k-l}{k}\pi\right),\quad \text{for all}~l = 0,\dots , k.
\end{equation*}
Introducing the basis consisting of the {\em Lagrange functions}
\begin{equation*}
L^k_i(x) \bydef \prod_{j\neq i}\frac{x-x^k_j}{x^k_i-x^k_j},
\end{equation*}
we have that the {\em Lagrange interpolation polynomial of order $k$} is given by
\begin{equation} \label{eq:Lagrange_polynomial}
P_k(u)(x)=\sum_{i=0}^k u(x^k_i)L^k_i(x).
\end{equation}
One can then show (see for instance~\cite{MR3012510}), that
\begin{equation}
\label{eq:formula_Lambdak}
\Lambda_k=\sup_{x\in[-1,1]} \sum_{i=0}^k \vert L^k_i(x)\vert,
\end{equation}
and therefore we get
\begin{equation*}
\sup_{x\in[-1,1]} \vert P_k(u)(x)\vert \leq \Lambda_k \max\limits_{0\leq i \leq k} \vert u(x^k_i) \vert,
\end{equation*}
which is exactly~\eqref{eq:bound_infty_from_max}.

Since we used several times~\eqref{eq:bound_infty_from_max} and Proposition~\ref{prop:interpolation_error2} in Section~\ref{sec:bounds}, the bounds that we obtained there depend on the Lebesgue constant $\Lambda_k$. Therefore we need computable (and as sharp as possible) upper bounds for $\Lambda_k$. One possibility is to use the well known bound (again see for instance~\cite{MR3012510})
\begin{equation}
\label{eq:bound_Lambdak}
\Lambda_k \leq 1+\frac{2}{\pi}\ln(k+1).
\end{equation}
However, we can do better, at least when $k$ is odd. In that case, it has been shown (see \cite{MR0194799}) that
\begin{equation*}
\Lambda_k = \frac{1}{k}\sum_{l=0}^{k-1}\cot \left( \frac{2l+1}{4k}\pi \right) ,
\end{equation*}
and this formula can be evaluated rigorously using interval arithmetic. Unfortunately, there is no such formula for $k$ even. For small values ($k=2$ and $k=4$) we computed $\Lambda_k$ by hand using~\eqref{eq:formula_Lambdak}, and for $k\geq 6$ we used~\eqref{eq:bound_Lambdak} (it is also know that $\Lambda_k \sim \frac{2}{\pi}\ln(k)$, therefore~\eqref{eq:bound_Lambdak} is sharp for large $k$).

We now turn our attention to the interpolation estimates of Section~\ref{sec:general}. We point out that the analogue of Proposition~\ref{prop:interpolation_error1} for the Chebyshev interpolation points of the first kind is very standard, and can be found in many textbooks. However, the case of the Chebyshev interpolation points of the second kind is seldom discussed, therefore we include a proof here for the sake of completeness (which is nothing but a slight adaptation of the \emph{standard} proof).

\medskip
\noindent \textit{Proof of Proposition~\ref{prop:interpolation_error1}.}
We consider the polynomial $W_k(x)=\prod_{l=0}^k (x-x_l^k)$ and use the standard interpolation error estimate for a function $u\in\CC^{k+1}$ (see for instance \cite{MR1656150}),
\begin{equation*}
\left\Vert u-P_k(u)\right\Vert_{\infty} \leq \frac{\left\Vert W_k\right\Vert_{\infty}}{(k+1)!}\left\Vert u^{(k+1)} \right\Vert_{\infty}.
\end{equation*}
To prove Proposition~\ref{prop:interpolation_error1}, we only have to show that $\left\Vert W_k\right\Vert_{\infty}= \frac{1}{2^{k-1}}$ (the remaining factor $\frac{1}{2^{k+1}}$ coming from the rescaling). Introducing, for $k\in\N$, the $k$-th Chebyshev polynomial of the second kind $U_k$, defined by
\begin{equation*}
U_k(\cos(\theta))=\frac{\sin(k\theta)}{\sin(\theta)},
\end{equation*} 
we have that
\begin{equation}
\label{def:W_k} 
W_k(x)=(x-1)(x+1)\frac{U_k(x)}{2^{k-1}}.
\end{equation}
Indeed, the right hand side of~\eqref{def:W_k} is a unitary polynomial of degree $k+1$, that has the same zeros as $W_k$. We can then rewrite
\begin{align*}
W_k(x) &= \frac{1}{2^{k-1}}(x-1)(x+1)\frac{\sin(k\arccos(x))}{\sqrt{1-x^2}} \\
&= -\frac{1}{2^{k-1}}\sqrt{1-x^2}\sin(k\arccos(x)),
\end{align*}
and we end up with
\begin{equation*}
\left\Vert W_k\right\Vert_{\infty} = \frac{1}{2^{k-1}},
\end{equation*}
so Proposition~\ref{prop:interpolation_error1} is proven. \hfill $\qed$

\medskip
\noindent \textit{Proof of Proposition~\ref{prop:interpolation_error2}.}
The first part of the bound, namely
\begin{equation*}
\left(1+\Lambda_k\right)\left(\frac{\pi}{4}\right)^l\frac{(k+1-l)!}{(k+1)!},
\end{equation*}
comes from a combination of the Lebesgue constant and Jackson's Theorem, and can be found in~\cite{MR1656150}. However, it does not give a very sharp interpolation error estimate for small values of $k$ and $l$, therefore we derive here the second part of the bound, namely
\begin{equation*}
\frac{1}{l!2^l}\sum_{q=0}^{\left[\frac{l-1}{2}\right]}\frac{1}{4^q}\binom{l-1}{2q}\binom{2q}{q}
\end{equation*}
that can be used in those cases. 

Letting $u(x)=x^p$ (with $p\in\{0,\ldots,k\}$) in \eqref{eq:Lagrange_polynomial} leads to
\begin{equation}
\label{eq:lagrange_basis}
x^p=\sum_{i=0}^k (x^k_i)^p L^k_i(x), \quad \text{for all}~x\in\R.
\end{equation}
We now fix a function $u\in\CC^{l}$. Using~\eqref{eq:lagrange_basis} with $p=0$ (that is $1=\sum_{i=0}^k L^k_i(x)$), we get
\begin{equation*}
P_k(u)(x)-u(x)=\sum_{i=0}^k u(x^k_i)L^k_i(x)-  u(x) \left( \sum_{i=0}^k L^k_i(x) \right) =\sum_{i=0}^k \left(u(x^k_i)-u(x)\right)L^k_i(x).
\end{equation*}
Using Taylor's formula, we then get
\begin{align*}
P_k(u)(x)-u(x) &= \sum_{i=0}^k \left(\sum_{p=1}^{l-1}\frac{(x^k_i-x)^p}{p!}u^{(p)}(x) + \frac{(x^k_i-x)^l}{l!}u^{(l)}(y_i) \right)L^k_i(x) \\
&= \sum_{p=1}^{l-1} \frac{u^{(p)}(x)}{p!}\sum_{i=0}^k (x^k_i-x)^pL^k_i(x) + \sum_{i=0}^k\frac{(x^k_i-x)^l}{l!}u^{(l)}(y_i)L^k_i(x),
\end{align*}
for some $y_i$ in $[-1,1]$. Then, expanding the $(x^k_i-x)^p$ terms and using again~\eqref{eq:lagrange_basis}, we get that
\begin{align*}
\sum_{i=0}^k (x^k_i-x)^pL^k_i(x) &= \sum_{q=0}^p\binom{p}{q}(-x)^{p-q}\sum_{i=0}^k (x^k_i)^qL^k_i(x) \\
&= \sum_{q=0}^p\binom{p}{q}(-x)^{p-q}x^q \\
&= (x-x)^p \\
&= 0,
\end{align*}
and thus
\begin{align*}
P_k(u)(x)-u(x) = \sum_{i=0}^k\frac{(x^k_i-x)^l}{l!}u^{(l)}(y_i)L^k_i(x),
\end{align*}
Letting 
\begin{equation*}
\lambda^k_i \bydef \prod_{j\neq i}\frac{1}{x^k_i-x^k_j},
\end{equation*}
we can easily observe that $L_i^k(x) = \lambda^k_i W_k(x)/(x-x_i^k)$, and therefore
\begin{align*}
\left\vert P_k(u)(x)-u(x)\right\vert &\leq \frac{\left\Vert u^{(l)}\right\Vert_{\infty}}{l!}\sum_{i=0}^k\vert x^k_i-x\vert^l\vert L^k_i(x)\vert \\
&= \frac{\left\Vert u^{(l)}\right\Vert_{\infty}}{l!}\left\vert W_k(x)\right\vert \sum_{i=0}^k\vert\lambda_i^k \vert \vert x^k_i-x\vert^{l-1}.
\end{align*}
According to \cite{MR3012510}, in case the points $x^k_i$ are the Chebyshev interpolation points of the second kind, we have
\[
\lambda^k_i = 
\begin{cases}
(-1)^i\frac{2^{k-1}}{k}, & i = 1,\dots,k-1,\\
(-1)^i\frac{2^{k-1}}{2k}, &i = 0,k.
\end{cases}
\]
Remembering that $\left\Vert W_k\right\Vert_{\infty}=\frac{1}{2^{k-1}}$, we get
\begin{equation*}
\left\vert P_k(u)(x)-u(x)\right\vert \leq \frac{\left\Vert u^{(l)}\right\Vert_{\infty}}{l!} \frac{1}{k} \left(\frac{(1+x)^{l-1}}{2}+\frac{(1-x)^{l-1}}{2} + \sum_{i=1}^{k-1} \vert x^k_i-x\vert^{l-1}\right).
\end{equation*}
The function 
\begin{equation*}
x\mapsto \frac{(1+x)^{l-1}}{2}+\frac{(1-x)^{l-1}}{2} + \sum_{i=1}^{k-1} \vert x^k_i-x\vert^{l-1}
\end{equation*} 
is even and increasing on $[0,1]$, therefore its maximum is reached at $x=1$ and we get
\begin{equation*}
\left\vert P_k(u)(x)-u(x)\right\vert \leq \frac{\left\Vert u^{(l)}\right\Vert_{\infty}}{l!} \frac{1}{k} \left(2^{l-2} + \sum_{i=1}^{k-1} \left(1-\cos \frac{i\pi}{k}\right)^{l-1}\right).
\end{equation*}
Then, we compute
\begin{align*}
2^{l-2} + \sum_{i=1}^{k-1} \left(1-\cos \frac{i\pi}{k}\right)^{l-1} & = \sum_{i=0}^{k} \left(1-\cos \frac{i\pi}{k}\right)^{l-1} - 2^{l-2} \\
& = \sum_{i=0}^{k} \sum_{q=0}^{l-1}\binom{l-1}{q}(-1)^q\cos^q \frac{i\pi}{k} - 2^{l-2} \\
& = \sum_{q=0}^{\left[\frac{l-1}{2}\right]}\binom{l-1}{2q}\sum_{i=0}^{k}\cos^{2q} \frac{i\pi}{k} - 2^{l-2} \\
& = \sum_{q=0}^{\left[\frac{l-1}{2}\right]}\binom{l-1}{2q}\sum_{i=0}^{k}\frac{1}{4^q}\left(\binom{2q}{q}+\sum_{j=0}^{q-1}\binom{2q}{j}\cos 2(q-j)\frac{i\pi}{k}\right) - 2^{l-2} \\
& = \sum_{q=0}^{\left[\frac{l-1}{2}\right]}\binom{l-1}{2q}\frac{1}{4^q}\left((k+1)\binom{2q}{q}+2\sum_{j=0}^{q-1}\binom{2q}{j}\right) - 2^{l-2} \\
& = \sum_{q=0}^{\left[\frac{l-1}{2}\right]}\binom{l-1}{2q}\frac{1}{4^q}\left(k\binom{2q}{q}+4^q\right) - 2^{l-2} \\
& = k\sum_{q=0}^{\left[\frac{l-1}{2}\right]}\binom{l-1}{2q}\binom{2q}{q}.
\end{align*}
We end up with
\begin{equation*}
\left\vert P_k(u)(x)-u(x)\right\vert \leq \frac{\left\Vert u^{(l)}\right\Vert_{\infty}}{l!} \sum_{q=0}^{\left[\frac{l-1}{2}\right]}\binom{l-1}{2q}\binom{2q}{q},
\end{equation*}
and Proposition~\ref{prop:interpolation_error2} is proven (the lacking $\frac{1}{2^l}$ factor coming from the time rescaling). \hfill $\qed$


\end{document}